\documentclass[11pt]{article}
\usepackage{amsmath,amssymb,amsthm,bbm,natbib,xcolor,hyperref,nccmath}
\usepackage[margin=1.2in]{geometry}
\usepackage{caption}
\usepackage{graphicx}
\usepackage{subcaption}
\usepackage{afterpage}
\usepackage[hang,flushmargin]{footmisc}
\usepackage{float}
\usepackage{cleveref}

\newtheorem{theorem}{Theorem}

\newtheorem{proposition}{Proposition}
\newtheorem{remark}{Remark}
\newtheorem{definition}{Definition}

\allowdisplaybreaks

\newcommand{\R}{\mathbb{R}}
\newcommand{\N}{\mathbb{N}}

\newcommand{\EE}[1]{\mathbb{E}\left[{#1}\right]}
\newcommand{\EEst}[2]{\mathbb{E}\left[{#1}\  \middle| \ {#2}\right]}
\newcommand{\Ep}[2]{\mathbb{E}_{{#1}}\left[{#2}\right]}

\newcommand{\PP}[1]{\mathbb{P}\left\{{#1}\right\}}
\newcommand{\PPst}[2]{\mathbb{P}\left\{{#1}\  \middle| \ {#2}\right\}}

\newcommand{\Pp}[2]{\mathbb{P}_{{#1}}\left\{{#2}\right\}}
\newcommand{\One}[1]{{\mathbbm{1}}\left\{{#1}\right\}}
\newcommand{\one}[1]{{\mathbbm{1}}_{{#1}}}
\newcommand{\iidsim}{\stackrel{\textnormal{iid}}{\sim}}

\newcommand\independent{\protect\mathpalette{\protect\independenT}{\perp}}
\def\independenT#1#2{\mathrel{\rlap{$#1#2$}\mkern2mu{#1#2}}}

\newcommand{\eqd}{\stackrel{\textnormal{d}}{=}}

\newcommand{\Ch}{\widehat{C}}

\newcommand{\muhat}{\widehat{\mu}}
\newcommand{\mutil}{\tilde{\mu}}
\newcommand{\test}{\textnormal{test}}

\usepackage{afterpage}
\usepackage{multirow}

\newcommand*\samethanks[1][\value{footnote}]{\footnotemark[#1]}

\title{Distribution-free inference with hierarchical data}
\author{Yonghoon Lee\thanks{Department of Statistics, University of Pennsylvania}, Rina Foygel Barber\thanks{Department of Statistics, University of Chicago}, and Rebecca Willett\samethanks\,\,\thanks{Department of Computer Science, University of Chicago}}
\date{}

\begin{document}
\maketitle

\begin{abstract}
This paper studies distribution-free inference in settings where the data set has a hierarchical structure---for example, groups of observations,
or repeated measurements. In such settings, standard notions of exchangeability may not hold. To address this challenge, a hierarchical form of exchangeability is derived, facilitating extensions of distribution-free methods, including conformal prediction and jackknife+.  While the standard theoretical guarantee obtained by 
the conformal prediction framework is a marginal predictive coverage guarantee, in the special case of independent repeated measurements, it is possible to  achieve a stronger form of 
coverage---the ``second-moment coverage'' property---to provide better control of conditional miscoverage rates, and distribution-free prediction sets that achieve this property are constructed. Simulations illustrate that this guarantee indeed leads to uniformly small conditional miscoverage rates. Empirically, this stronger guarantee comes at the cost of a larger width of the prediction set in scenarios where the fitted model is poorly calibrated, but this cost is very mild in cases where the fitted model is accurate.
\end{abstract}

\section{Introduction}
Consider a prediction problem where we have training data $\{(X_i,Y_i)\}_{i=1,2,\dots,n}$, and a new observation $X_{\test}$ for
which we need to predict its response, $Y_{\test}$. The task of \emph{distribution-free prediction} is to construct
a prediction band $\Ch$ satisfying
\begin{equation}\label{eqn:marginal_coverage_kuarantee}
\PP{Y_{\test} \in \Ch(X_{\test})} \geq 1-\alpha
\end{equation}
with respect to data $(X_1,Y_1),\dots,(X_n,Y_n),(X_{\test},Y_{\test})\iidsim P$, for \emph{any} distribution $P$. Methods such as conformal prediction \citep{vovk2005algorithmic} provide an answer to this problem whose validity does not depend on any assumptions on $P$---and indeed, validity holds more generally for any exchangeable data distribution (with i.i.d.\ data
points being an important special case).

Though the marginal coverage guarantee~\eqref{eqn:marginal_coverage_kuarantee} guarantees an overall quality of the prediction set, it is often more desired to have a useful guarantee with conditional coverage
\[\PPst{Y_{\test} \in \Ch(X_{\test})}{X_{\test}},\]
to ensure that the prediction is accurate conditional on the specific feature values of the test point. However, achieving distribution-free conditional coverage is a much more challenging problem. In case of nonatomic features (i.e., where the distribution of $X_i$
has no point masses, such as a continuous distribution),  \citet{vovk2005algorithmic,lei2014distribution,barber2021limits} prove fundamental limits on the possibility of constructing this stronger type of prediction interval---indeed, in the case of a real-valued response $Y$, any such $\Ch$ cannot have finite width.

In this work, we study a different setting where the data sampling mechanism is hierarchical, rather than i.i.d.---this builds on the work
of \citet{dunn2022distribution} (detailed below), who also studied conformal prediction in the hierarchical setting. 
We  develop an extension of conformal prediction that works for such data structure, which we denote as hierarchical conformal prediction (HCP),
to provide distribution-free marginal predictive coverage in this non-i.i.d.\ setting. As detailed later, our proposed method includes some of the methods discussed in \citet{dunn2022distribution} as special cases, while also providing a theoretically tight coverage guarantee.
Furthermore, we also explore an important special case, where our data set
contains multiple observations (multiple draws of $Y$) for each individual (for each draw of $X$); in
this setting, we propose a variant of our method, HCP$^2$,
that moves towards conditional coverage via a stronger ``second-moment coverage'' guarantee.

\subsection{Problem setting: hierarchical sampling}\label{sec:intro_setting}

In an i.i.d.\ setting, we would work with data points that are sampled as $Z_i = (X_i,Y_i)\iidsim P$ for some distribution $P$.
For example, each data point $i$ might correspond to a randomly sampled individual; we would like to ask 
how features $X_i$
can be used to predict a response $Y_i$---for instance, if we are studying academic performance for a sample of children,
$X_i$ might reflect features such as age, school quality rating, parents' income level, etc, for the $i$th child in the sample,
 and $Y_i$ is the child's test score.

In a hierarchical setting, we instead assume that the data are drawn via a hierarchical sampling procedure.
For example, if our procedure for recruiting subjects in our study involves sampling classrooms within a city,
and then sampling multiple children within each classroom, then we would 
expect some amount of dependence between subjects that are in the same classroom.
Can we still use this training data set to build a model for predicting $Y$ from $X$, for a new child
that comes from a new (previously unsampled) classroom?
In this setting, statistical procedures that rely on an i.i.d.\ sample may no longer be valid. Parametric approaches 
might address this issue by, e.g., adding a random effects component to the parametric model.
In this work, we instead seek to adapt to hierarchically sampled data within the framework of distribution-free inference.

\paragraph{Hierarchical i.i.d.\ sampling.}
To generalize the example above, suppose
 the training sample consists of $K\geq 1$ many \emph{groups} of observations, where group $k$ 
contains $N_k\geq 1$ many data points $Z_{k,1},\dots,Z_{k,N_k}$. 
We can consider a hierarchical version of the i.i.d.\ sampling assumption: first we sample the distributions and sample sizes 
that characterize each group,
\[\Pi_k \sim P_\Pi, \ N_k\mid \Pi_k \sim P_{N\mid\Pi}, \textnormal{ independently for each $k=1,\dots,K$,}\]
where $P_\Pi$ is a distribution over a measurable  space $\mathcal{P}(\mathcal{Z})$ of probability measures on $\mathcal{Z}$.
$P_{N\mid \Pi}$ is a conditional distribution on the within-group sample size $N\in \mathbb{N}=\{1,2,\dots\}$.
Then conditional on this draw, we sample observations within each group,
\[Z_{k,1},\dots,Z_{k,N_k}\mid (\Pi_k,N_k)\, \iidsim \Pi_k, \textnormal{ independently for each $k=1,\dots,K$}.\]
In this setting, we aim to provide a prediction interval\footnote{In general, predictive inference can provide a \emph{prediction set}
that may or may not be an interval---and indeed, if the response $Y$ is not real-valued, it may not be meaningful to ask for an interval
of values. However, we  continue to refer to the ``prediction interval'' for simplicity.} for the unobserved test outcome $Y_\test$ given a new feature input $X_\test$ from a new group, where the test point $Z_{\test} = (X_{\test},Y_{\test})$ is sampled as
\[\Pi_{K+1} \sim P_\Pi,\qquad \ Z_{\test} \mid \Pi_{K+1} \sim \Pi_{K+1},\]
independently of the training data.

A key observation is that this prediction task can equivalently be characterized as follows.
Imagine that we instead sample a collection of data points from $K+1$, rather than $K$,
many groups:
\begin{equation}\label{eqn:hier_iid}
\Pi_k \sim P_\Pi, \ N_k\mid \Pi_k \sim P_{N\mid \Pi}, \ Z_{k,1},\dots,Z_{k,N_k}\mid (\Pi_k,N_k)\iidsim \Pi_k,\end{equation}
independently for each $k=1,\dots,K+1$. Here groups $k=1,\dots,K$ correspond to training data.
Then our prediction task can equivalently be characterized
as the task of predictive inference for $Z_{K+1,1}$ (that is, writing $Z_{K+1,1} = (X_{K+1,1},Y_{K+1,1})$,
then based on features $X_{K+1,1}$ we need to construct a prediction interval for $Y_{K+1,1}$)---the remaining test points $Z_{K+1,2}, \dots, Z_{K+1,N_{K+1}}$ can be regarded as unobserved.

\paragraph{Defining hierarchical exchangeability.}
We can generalize the above construction to a hierarchical notion of exchangeability, which we formalize as follows.
First, let $\mathcal{Z} = \mathcal{X}\times \mathcal{Y}$ (where $\mathcal{X}$ and $\mathcal{Y}$ denote the spaces
in which the features $X$ and responses $Y$ lie, respectively), and define
\[\tilde{\mathcal{Z}} = \mathcal{Z}\cup \mathcal{Z}^2 \cup\mathcal{Z}^3\cup \dots,\]
the set of \emph{all sequences of any finite length} with entries lying in $\mathcal{Z}$. Here, $\mathcal{Z}^j$ denotes the set of length-$j$ vectors whose elements belong to $\mathcal{Z}$, for $j = 1, 2, \ldots$.
With this definition in place, we can define $\tilde{Z}_k = (Z_{k,1},\dots,Z_{k,N_k}) \in\tilde{\mathcal{Z}}$,
the $k$th group of observations within our sample. For any $\tilde{z}\in\tilde{\mathcal{Z}}$, we also define
$\textnormal{length}(\tilde{z})\in\mathbb{N}$ as the length of the finite sequence $\tilde{z}$.

Recall that the standard definition of exchangeability for a data set $(Z_1,\dots,Z_{n+1})$, is given by
\begin{equation}\label{eqn:define_exch}(Z_1,\dots,Z_{n+1}) \eqd (Z_{\sigma(1)},\dots,Z_{\sigma(n+1)})\textnormal{ for all $\sigma\in\mathcal{S}_{n+1}$}.\end{equation}
Hierarchical exchangeability extends this condition as follows.

\begin{definition}[Hierarchical exchangeability]\label{def:hier_exch}
Let $\tilde{Z}_1,\dots,\tilde{Z}_{K+1} \in \tilde{\mathcal{Z}}$ be a sequence of random variables with a joint distribution.
We say that this sequence (or equivalently, its distribution) satisfies \emph{hierarchical exchangeability}
if, first, 
\[(\tilde{Z}_1,\dots,\tilde{Z}_{K+1}) \eqd (\tilde{Z}_{\sigma(1)},\dots,\tilde{Z}_{\sigma(K+1)})\]
 for all $\sigma\in\mathcal{S}_{K+1}$,
and, second, 
\[(\tilde{Z}_1,\dots,\tilde{Z}_k,\dots,\tilde{Z}_{K+1}) \eqd (\tilde{Z}_1,\dots,(\tilde{Z}_{k,\sigma(1)},\dots,\tilde{Z}_{k,\sigma(m)}),\dots,\tilde{Z}_{K+1}) \mid \textnormal{length}(\tilde{Z}_k)=m\]
 for all $k=1,\dots,K+1$, all  $m\geq 1$ for which $\PP{\textnormal{length}(\tilde{Z}_k)=m}>0$, and
all $\sigma\in\mathcal{S}_m$.
\end{definition}
\noindent (Here $\mathcal{S}_r$ denotes the set of permutations on $\{1,\dots,r\}$, for any $r\in\mathbb{N}$.) In other words, hierarchical exchangeability requires that, first, the groups are exchangeable with each other, and second,
the observations within each group are exchangeable as well. The hierarchical exchangeability condition generalizes the hierarchical i.i.d. assumption and allows for application to broader settings. For example, in a medical context, each data point may correspond to a patient, while groups represent different hospitals. The hierarchical exchangeability condition accounts for scenarios where patients within the same hospital may share similarities or have dependent data---for instance, due to shared environmental factors from staying in the same facility---and also accommodates situations where observations from different hospitals are not independent, such as when there are shared influences like the spread of a disease or the development of treatments.

\paragraph{Special case: equal group sizes.} If the group sizes are all equal, $N_1 = \dots = N_{K+1}=: N$, then the definition of hierarchical exchangeability can be interpreted in another way: the distribution of the data points $(Z_{1,1},\dots,Z_{1,N},\dots,Z_{K+1,1},\dots,Z_{K+1,N})$ is invariant to all permutations $\sigma\in\mathcal{S}_{(K+1)N}$ that preserve groups (that is, any two data values belonging to the same group, should again belong to the same group after permutation). This type of exchangeability is an instance of the methodology developed by work such as \citet{chernozhukov2018exact}, which consider predictive inference for distributions that are invariant to certain subgroups of permutations. In general, however, when the vectors of observations from each group may be of different lengths (as in Definition~\ref{def:hier_exch}), hierarchical exchangeability cannot be represented via a subgroup of permutations.

\paragraph{Special case: repeated measurements.} Another special case of interest is the setting where we have data with repeated measurements. In this setting, given an unknown distribution $P$ on $\mathcal{Z}=\mathcal{X}\times\mathcal{Y}$, writing $P_X$ and $P_{Y|X}$ as the corresponding marginal and conditional distribution,
if for example we assume each batch of measurements has a common fixed size $N$, we draw the training data as
\begin{equation}\label{eqn:intro_repeated}\begin{cases} X_k\sim P_X, \\ Y_{k,1},\dots,Y_{k,N}\mid X_k\iidsim P_{Y|X},\end{cases}\end{equation}
independently for each $k=1,\dots,K$. The test point is then given by an independent draw $(X_{\test},Y_{\test})\sim P$.
Equivalently, we can formulate the problem as sampling $K+1$, rather than $K$, many groups from the distribution~\eqref{eqn:intro_repeated},
with groups $k=1,\dots,K$ comprising the training data, and with $(X_{K+1}, Y_{K+1,1})$ determining the test point.
By defining $Z_{k,i} = (X_k,Y_{k,i})$, we can see that this construction satisfies the hierarchical exchangeability property given in Definition~\ref{def:hier_exch} (and indeed, can be viewed as a special case of the hierarchical-i.i.d.\ construction
given in~\eqref{eqn:hier_iid}, by taking $\Pi_k$ to be the distribution of $(X,Y)$ conditional on $X=X_k$).

\paragraph{Contributions.}  The main contributions of the present work can be summarized as follows. First, 
for hierarchical data $\tilde{Z}_1,\dots,\tilde{Z}_{K+1}$ satisfying the hierarchical exchangeability 
condition in Definition~\ref{def:hier_exch}, 
with $\tilde{Z}_1,\dots,\tilde{Z}_K$ comprising the training data while $\tilde{Z}_{K+1,1}$ is the target test point,
we provide a predictive inference method (HCP) that
offers a marginal coverage guarantee. Moreover,
for the special case of repeated measurements,
we provide a predictive inference method (HCP$^2$)
that offers a stronger notion of distribution-free validity (moving towards conditional, rather than marginal, coverage).

\subsection{Related work}
Distribution-free inference has received much attention recently in the statistics and machine learning literature,
with many methodological advances but also many results examining
the limits of the distribution-free framework.

In the standard setting with exchangeable (e.g., i.i.d.) data, conformal prediction (see \citet{vovk2005algorithmic, lei2018distribution} for background) provides a universal framework for distribution-free prediction. Split conformal prediction \citep{vovk2005algorithmic, papadopoulos2002inductive} offers a variant of conformal prediction that relies on data splitting to reduce the computational cost, but with some loss of statistical accuracy. \citet{barber2021predictive} introduces the jackknife+, which provides a valid distribution-free prediction set without data splitting while having a more moderate computational cost than full conformal prediction.

Applications of the conformal prediction framework to settings beyond the framework of exchangeability have also been studied in the recent
 literature. For example, \citet{tibshirani2019conformal} and \citet{podkopaev2021distribution} introduce extensions of conformal prediction that deal with covariate shift and label shift, respectively, while \citet{barber2022conformal} allow for unknown deviations from exchangeability.  \citet{xu2021conformal} considers applications to time series data, and \citet{lei2015conformal} considers functional data.

For the goal of having a useful bound for conditional coverage in the distribution-free setting, \citet{vovk2012conditional} and \citet{barber2021limits} show impossibility results for coverage conditional on $X_{\test}$ in the case that the feature distribution is nonatomic. \citet{izbicki2019flexible} and~\citet{chernozhukov2021distributional} propose methods that provide  conditional coverage guarantees under strong additional assumptions. Conformalized quantile regression, introduced by~\citet{romano2019conformalized}, can provide improved control of conditional miscoverage rates empirically, by accounting for heterogeneous spread of $Y$ given $X$.

Problems other than prediction have been studied as well in the distribution-free framework. \citet{barber2020distribution} and \citet{medarametla2021distribution} discuss limits on having a useful  inference on conditional mean and median, respectively. \citet{lee2021distribution} prove that for discrete feature distributions, we can make use of repeated feature observations to attain meaningful inference for the conditional mean.

Distribution-free inference for the hierarchical data setting was previously studied by \citep{dunn2022distribution}, proposing a range of different methods
for the hierarchical setting. We  describe their methods and results in detail in Section~\ref{sec:Dunn}.
See also~\citet{duchi2024predictive} for extensions of predictive inference methods in hierarchical settings.
The setting of repeated measurements is  explored by \citet{cheng2022many}, who contrast empirical risk minimization when the repeated measurements are aggregated prior to estimation and when they are kept separate. Our contributions complement their work with a framework for conformal prediction. Like \citet{cheng2022many}, we find that non-aggregated repeated measurements can enable
 a more informative statistical analysis. Finally, as mentioned above, the methodology we develop is related to the idea of invariance of the distribution with respect to certain sets of permutations as studied by \citet{chernozhukov2018exact}, in the special case of equal-size groups.\footnote{The work of \citet{dobriban2025symmpi}, which is more recent than this present paper, generalizes this idea to distributions that are invariant under a group action, and points out that our definition of hierarchical exchangeability is a special case of their approach.}

\subsection{Outline}
The remainder of this paper is organized as follows. In Section~\ref{sec:HCP}, we first give background on existing
methods for the hierarchical setting \citep{dunn2022distribution}, and then present our proposed method, Hierarchical Conformal Prediction (HCP),
and compare HCP with the existing methods on simulated data.
In Section~\ref{sec:HCP2}, we turn to the special case of data with repeated measurements; we examine the problem
of conditionally valid coverage in this setting, and propose a variant of our method, HCP$^2$, that offers a stronger coverage guarantee,
which we verify with simulations. Section~\ref{sec:discussion} offers a short discussion of our findings and of open questions raised by this work.
The proofs of our theorems, together with some additional methods and theoretical results, are deferred to the Appendix.

\section{Hierarchical conformal prediction}\label{sec:HCP}

In this section, we give a short introduction to the well-known split conformal method for i.i.d.\ (or exchangeable) data,
and then propose our new method, hierarchical conformal prediction (HCP), which adapts split conformal to the setting
of hierarchically structured data.

\subsection{Background: split conformal}
First, for background, consider the setting of i.i.d.\ or exchangeable data, $Z_1,\dots,Z_n,Z_{n+1}$,
where $Z_i=(X_i,Y_i)$ for each $i$. Here $Z_1,\dots,Z_n$ is the training data, while $Z_{\test}=Z_{n+1}$ is the test point---with only $X_{n+1}$ observed.
The \emph{split conformal prediction} method \citep{vovk2005algorithmic}
constructs a distribution-free prediction interval as follows. As a preliminary step, we split the training data into two subsets of size $n_0 + n_1 = n$, where, without loss of generality, the first $n_0$ data points are used for training, and the remaining $n_1$ data points are used for calibration. Once the data is split, the first step is to use the training portion of the data (i.e., $Z_1,\dots,Z_{n_0}$), we fit
a \emph{score function} 
\[s: \mathcal{Z}\rightarrow\R,\]
where $s(z) = s(x,y)$ measures the extent to which a data point $(x,y)$ \emph{conforms} to the trends observed in the training data,
with large values indicating a more unusual value of the data point. For example, if the response variables $Y$
lie in $\mathcal{Y}=\R$, one mechanism for defining this function
is to first run a regression on the labeled data set $\{(X_i,Y_i):i=1,\dots,n_0\}$, to produce a fitted model $\widehat{\mu}:\mathcal{X}\rightarrow\R$,
and then define $s(x,y) = |y - \widehat{\mu}(x)|$, the absolute value of the residual. 
Next, the second step is to use the calibration set (i.e., $Z_{n_0+1},\dots,Z_n$) to determine a cutoff threshold for the score, 
and construct the corresponding prediction interval:
\begin{equation}\label{eqn:define_C_hat_splitCP}\Ch(X_{\test}) = \left\{y\in\mathcal{Y} : s(X_{\test},y) \leq Q_{1-\alpha}\left(\sum_{i=n_0+1}^n \frac{1}{n_1+1}\delta_{s(Z_i)} + \frac{1}{n_1+1}\delta_{+\infty}\right)\right\},\end{equation}
where $\delta_t$ denotes the point mass at $t$, while $Q_{1-\alpha}(\cdot)$ denotes the $(1-\alpha)$-quantile of a distribution.\footnote{Formally, for a distribution $P$ on $\R$,
the quantile is defined as $Q_{1-\alpha}(P) = \inf\{t\in\R: \Pp{T\sim P}{T \leq t} \geq 1-\alpha\}$.}
\sloppypar \citet{vovk2005algorithmic} establish marginal coverage for the split conformal method:
if the data $Z_1,\dots,Z_n,Z_{\test}$ satisfies exchangeability (that is, if the property~\eqref{eqn:define_exch}
holds when we take $Z_{\test}=Z_{n+1}$),
then
\[\PP{Y_{\test}\in\Ch(X_{\test})}\geq 1-\alpha.\]
We note that this guarantee is marginal---in particular, it does not hold if we were to condition on the test features $X_{\test}$ (and also, 
holds only on average over the draw
of the training and calibration data $Z_1,\dots,Z_n$).

In the setting of i.i.d.\ or exchangeable data, the full conformal prediction~\citep{vovk2005algorithmic} provides
an alternative method that avoids the loss of efficiency incurred by sample splitting, at the cost of a greater computational burden.
As a compromise between split and full conformal, cross-validation type methods are proposed by \citet{vovk2015cross,vovk2018cross,barber2021predictive}. We give background on these alternative
methods in Appendix.

\subsection{Proposed method: HCP} 
Now we return to the hierarchical data setting that is the problem of interest in this paper.
Can split conformal prediction be applied to this setting? That is, given the observed data points 
$Z_{1,1},\dots,Z_{1,N_1}, \dots, Z_{K,1},\dots,Z_{K,N_K}$, can we simply apply split conformal to this training data
set to obtain a valid prediction interval? In fact, the hierarchical exchangeability property (Definition~\ref{def:hier_exch})
is not sufficient.
In general,
the data may satisfy hierarchical exchangeability without satisfying the (ordinary, non-hierarchical) exchangeability property~\eqref{eqn:define_exch}
that is needed for split conformal to have theoretical validity---for example, 
if the data consists of students sampled from $K$ many classrooms, then there may be higher correlation between
students in the same classroom.
 An exception is the trivial special case where $N_1=\dots=N_K=1$, 
i.e., all groups contain exactly one observation, in which case hierarchical exchangeability simply reduces to ordinary exchangeability (and thus standard conformal prediction would ensure marginal predictive coverage at the desired level). At the other extreme, if the group sizes are highly imbalanced, and the predictive model's accuracy differs widely across groups, then we might expect that standard conformal prediction could fail to achieve the desired coverage level.

We now introduce our extension of split conformal prediction for the hierarchical data setting. 
Given training data $\tilde{Z}_1,\dots,\tilde{Z}_K$, we split it into $K_0+K_1$ many groups, with $K_0$ groups used for training
and $K_1=K-K_0$ groups for calibration. First, using data $\tilde{Z}_1,\dots,\tilde{Z}_{K_0}$, we fit a score function $s:\mathcal{Z}\rightarrow\R$
(as before, a canonical choice is to use the absolute value of the residual for some fitted model, $s(z) = s(x,y) = |y-\widehat{\mu}(x)|$
for some fitted model $\widehat{\mu}$). Next, we use the calibration set to set a threshold $T$ for the prediction interval, as follows:
\begin{multline}\label{eqn:define_HCP}\textnormal{HCP method: \quad }\Ch(X_{\test})= \{y\in\mathcal{Y} : s(X_{\test},y) \leq T\}, \textnormal{ where }\\
T = Q_{1-\alpha}\left(\sum_{k=K_0+1}^K\sum_{i=1}^{N_k}\frac{1}{(K_1+1)N_k}\cdot\delta_{s(Z_{k,i})}+\frac{1}{K_1+1}\cdot\delta_{+\infty}\right).\end{multline}
Compared to the original split conformal interval constructed in~\eqref{eqn:define_C_hat_splitCP}, we see
that the constructions follow the same format, but the score threshold for HCP  in general is different,
since this new construction places different amounts of weight on different training data points, depending on the sizes
of the various groups. However, as a special case, if $N_k=1$ for all $k$ (i.e., one observation in each group), then HCP
simply reduces to the split conformal method.

The HCP prediction interval offers the following distribution-free guarantee:
\begin{theorem}[Marginal coverage for HCP]\label{thm:HCP}
Suppose that $\tilde{Z}_1,\dots,\tilde{Z}_{K+1}$ satisfies the hierarchical exchangeability property
(Definition~\ref{def:hier_exch}), and suppose we run HCP with training data $\tilde{Z}_1,\dots,\tilde{Z}_{K_0}$, calibration data $\tilde{Z}_{K_0+1},\dots,\tilde{Z}_{K}$, and test 
point $Z_{\test}=(X_{\test},Y_{\test})=\tilde{Z}_{K+1,1}$. Then
\[\PP{Y_{\test}\in\Ch(X_{\test})}\geq 1-\alpha.\]
Moreover, if the scores $s(Z_{k,i})$ are distinct almost surely,\footnote{In fact, we  see in the proof that it suffices
to assume that $s(Z_{k,i})\neq s(Z_{k',i'})$ for all $k\neq k'$, almost surely---in other words, we can allow non-distinct scores within a group $k$, as long
as scores from different groups are always distinct.} it additionally holds that
\[\PP{Y_{\test}\in\Ch(X_{\test})} \leq 1-\alpha+\frac{2}{K_1+1},\]
where $K_1 = K - K_0$.
\end{theorem}

In Appendix, we provide analogous hierarchical versions of the jackknife+ and full conformal methods, which enables inference without sample splitting.

\begin{remark}
    If a few observations from the test group are also available and the goal is to make inference on a new sample from that group, one can apply HCP using a data-dependent score, without requiring sample splitting. For example, suppose all the group sizes are equal to $N$, and we observe $M$ outcomes in the test group. Then instead of splitting the data and using one to construct a score function, we can construct the following hierarchically exchangeable scores:
    \[\left(\left|Y_{k,1} - \frac{1}{M}\sum_{l=1}^M Y_{k,N-M+l}\right|,\left|Y_{k,2} - \frac{1}{M}\sum_{l=1}^M Y_{k,N-M+l}\right|,\cdots, \left|Y_{k,N-M} - \frac{1}{M}\sum_{l=1}^M Y_{k,N-M+l}\right|\right)_{1 \leq k \leq K}.\]
    Then the prediction set can be constructed as in~\eqref{eqn:define_HCP}, with 
    \[s(X_\test,y) = \left|y - \frac{1}{M}\sum (\text{observed outcomes from the test group})\right|.\]
    A discussion of other potential methods for this setting can be found in~\citet{dunn2022distribution}.
    
\end{remark}

\subsection{Comparison with \texorpdfstring{\citet{dunn2022distribution}}{}'s methods}\label{sec:Dunn}
In this section, we summarize the work of \citet{dunn2022distribution} for the hierarchical data setting and compare
to our proposed HCP method.
\citet{dunn2022distribution} introduce four approaches for distribution-free inference with hierarchical data, which they refer
to as Pooling CDFs, Double Conformal, Subsampling Once, and Repeated Subsampling.
Here we only consider the split conformal versions of their methods, in the context of the supervised learning 
problem (i.e., predicting $Y$ from $X$), to be directly comparable to the setting of our work.

In the remainder of this section, we rewrite their methods to be consistent with our notation and problem setting for ease of comparison. As we detail below, the comparison between HCP and  \citet{dunn2022distribution}'s four methods reveals:\footnote{Note, however, that the assumption in this work---hierarchical exchangeability---differs from the setting considered in \citet{dunn2022distribution}, and thus a direct comparison between the methods is not entirely appropriate.}
\begin{itemize}
\item \citet{dunn2022distribution}'s theoretical guarantee for Pooling CDFs offers only an asymptotic result. In Proposition~\ref{prop:pooling_cdfs}, we show that Pooling CDFs is actually exactly equivalent to running HCP but with a slightly
higher error level $\alpha$, and consequently, establish a novel finite-sample coverage guarantee.
\item Double Conformal was shown by  \citet{dunn2022distribution}  to have finite-sample coverage at level $1-\alpha$,
but in practice is substantially overly conservative.
\item Subsampling Once was also shown  by  \citet{dunn2022distribution}  to have finite-sample coverage at level $1-\alpha$,
but in practice can lead to more variable performance due to discarding most of the calibration data.
\item \citet{dunn2022distribution}'s theoretical guarantee for Repeated Subsampling establishes coverage at the weaker level $1-2\alpha$, even though empirical performance shows that it typically achieves $1-\alpha$ coverage. In our work (Proposition~\ref{prop:repeated_subsampling}), we establish
that  Repeated Subsampling is equivalent to running HCP on a bootstrapped version of the original calibration data set, and consequently, obtain
a novel guarantee
of  finite-sample coverage at level $1-\alpha$.
\end{itemize}
Overall, our proposed framework of hierarchical exchangeability allows us to build a new understanding of the finite-sample performance
of two of \citet{dunn2022distribution}'s methods, Pooling CDFs and Repeated Subsampling---that is, the two whose empirical performance was observed to be the best, without high variance or overcoverage.

As for HCP,
for all of \citet{dunn2022distribution}'s methods the first step is to use the training portion of the data (groups $k=1,\dots,K_0$) 
to train a score function $s:\mathcal{Z}\rightarrow\R$, and the next step is to define a prediction interval
$\Ch(X_{\test}) = \left\{y\in\mathcal{Y}:s(X_{\test},y) \leq T\right\}$, for some threshold $T$ determined by the calibration
set. We  next give details for how $T$ is selected for each of their proposed methods.

\subsubsection{Pooling CDFs}
This method estimates the conditional cumulative distribution function (CDF)
of the scores within each group in the calibration set, and then averages these CDFs
to estimate the distribution of the test point score. For each calibration group $k=K_0+1,\dots,K$,
write the empirical CDF for group $k$ by
\[\widehat{F}_k(t) = \frac{1}{N_k}\sum_{i=1}^{N_k}\One{s(Z_{k,i}) \leq t},\]
and then define the pooled (or averaged) CDF as 
\[\widehat{F}_{\text{pooled}}(t) = \frac{1}{K_1}\sum_{k=K_0+1}^K \widehat{F}_k(t).\]
The threshold $T$ is then taken to be the $(1-\alpha)$-quantile of this pooled CDF,
\[T = \inf\left\{ t: \widehat{F}_{\text{pooled}}(t) \geq 1-\alpha\right\}.\]
\citet{dunn2022distribution} show that this method provides an asymptotic $(1-\alpha)$  coverage guarantee
as $K_1\rightarrow\infty$, but there is no finite-sample guarantee.
Here, we show that our HCP framework allows for a stronger, finite-sample guarantee,
simply by reinterpreting the method as a version of HCP.
\begin{proposition}\label{prop:pooling_cdfs}
The Pooling CDFs method of \citet{dunn2022distribution}  is equivalent to running HCP with $\alpha' =\alpha + \frac{1-\alpha}{K_1 + 1}$ in place of $\alpha$.
Consequently, under hierarchical exchangeability
(Definition~\ref{def:hier_exch}), the Pooling CDFs method satisfies
\[\PP{Y_{\test}\in\Ch(X_{\test})}\geq 1-\alpha -  \frac{1-\alpha}{K_1 + 1}.\]
Moreover, if the scores $s(Z_{k,i})$ are distinct almost surely,
it additionally holds that
\[\PP{Y_{\test}\in\Ch(X_{\test})} \leq 1-\alpha + \frac{1+\alpha}{K_1 + 1}.\]
\end{proposition}
\begin{proof}
Examining the definition of $\widehat{F}_{\textnormal{pooled}}$, we can see that the threshold $T$ for Pooling CDFs can equivalently
be written as
\begin{align*}
T&= \inf\left\{t\in\R : \sum_{k=K_0+1}^K\sum_{i=1}^{N_k}\frac{1}{K_1N_k}\cdot\one{s(Z_{k,i})\leq t} \geq 1-\alpha\right\}\\
&= \inf\left\{t\in\R : \sum_{k=K_0+1}^K\sum_{i=1}^{N_k}\frac{1}{(K_1+1)N_k}\cdot\one{s(Z_{k,i})\leq t} \geq 1-\alpha'\right\}\\
&= Q_{1-\alpha'}\left(\sum_{k=K_0+1}^K\sum_{i=1}^{N_k}\frac{1}{(K_1+1)N_k}\cdot\delta_{s(Z_{k,i})}+\frac{1}{K_1+1}\cdot\delta_{+\infty}\right),
\end{align*}
where the second equality holds since $1-\alpha' = \frac{K_1}{K_1+1}\cdot (1-\alpha)$.
Note that the last expression coincides with the definition of HCP with $\alpha'$ in place of $\alpha$.
The remaining results follow directly from Theorem~\ref{thm:HCP}.
\end{proof}
This result implies the asymptotic result of \citet{dunn2022distribution} (since as $K_1\rightarrow\infty$,
this coverage guarantee approaches $1-\alpha$), and so 
this strengthens the existing result. 
Note, however, that the asymptotic result for Pooling CDFs does not require exchangeability of the group sizes $N_k$,
so \citet{dunn2022distribution}'s assumptions are slightly weaker than for the finite-sample guarantee of HCP.

\subsubsection{Double Conformal}
\citet{dunn2022distribution}'s second approach, Double Conformal, is designed for the special case where all group sizes are equal, $N_k\equiv N$.
It works by  quantifying the within-group uncertainties for each group and then combining them to calculate a threshold $T$ for the prediction set. 
Specifically, define
\[q_k = Q_{1-\alpha/2}\left(\sum_{i=1}^N\frac{1}{N+1}\delta_{s(Z_{k,i})} + \frac{1}{N+1}\delta_{+\infty}\right)\]
as the $(1-\alpha/2)$-quantile of the scores within group $k$ (with a small weight placed on $+\infty$ to act as
a correction for finite sample size---note that this is equivalent to running
split conformal~\eqref{eqn:define_C_hat_splitCP} at coverage level $1-\alpha/2$ 
for the data in the $k$th group). Then define $\Ch(X_{\test}) =  \left\{y\in\mathcal{Y}:s(X_{\test},y) \leq T\right\}$
for a threshold $T$ defined as
\[T = Q_{1-\alpha/2}\left(\sum_{k=K_0+1}^{K}\frac{1}{K_1 + 1}\delta_{q_k} + \frac{1}{K_1+1}\delta_{+\infty}\right).\]
\citet{dunn2022distribution} show that the Double Conformal method provides finite-sample marginal coverage
at level $1-\alpha$ (unlike Pooling CDFs which offers only asymptotic coverage). 
However, in practice, the method is often substantially overly conservative---this is because the coverage
guarantee is proven via a union bound, with total error rate $\alpha$ obtained
by summing an $\alpha/2$ within-group error bound plus an $\alpha/2$ across-group error bound. 

\subsubsection{Subsampling Once}
Next, \citet{dunn2022distribution} propose Subsampling Once, an approach that avoids the issue of hierarchical exchangeability
by subsampling a single observation from each group, thus yielding a data set that satisfies (ordinary, non-hierarchical) exchangeability.\footnote{While
\citet{dunn2022distribution} present this method as a modification of full conformal, here for consistency of the comparison across different methods,
we use its split conformal version.}
Specifically, let $i_k$ be drawn uniformly from $\{1,\dots,N_k\}$, for each $k=K_0+1,\dots,K$. Then 
the subset of the calibration data given by $Z_{K_0+1,i_{K_0+1}},\dots,Z_{K,i_K}$, together with the test point $Z_{\test}$,
 now forms an exchangeable data set, and thus ordinary split conformal can be applied.
The prediction set is then defined with the threshold
\[T = Q_{1-\alpha}\left(\sum_{k=K_0+1}^K \frac{1}{K_1+1}\delta_{s(Z_{k,i_k})} + \frac{1}{K_1+1}\delta_{+\infty}\right).\]

 \citet{dunn2022distribution} show that, due to the exchangeability of this data subset, this procedure again offers 
 finite-sample marginal coverage
at level $1-\alpha$. Moreover, unlike the Double Conformal method, this last method is no longer overly conservative in practice.
As for Pooling CDFs, we note that this method does not require exchangeability of the group sizes $N_k$,
so the assumptions are slightly weaker than for the finite-sample guarantee of HCP.
However, the prediction interval constructed by the Subsampling Once method may be highly variable because much
of the available calibration data is discarded.

\subsubsection{Repeated Subsampling}
 To alleviate the problem of discarding data in the Subsampling Once method,
  \citet{dunn2022distribution} also provide an alternative method that makes use of multiple repeats of this procedure (Repeated Subsampling),
 and aggregating the outputs to provide a single prediction interval.
For each $b=1,\dots,B$,  randomly
choose a data point $i_k^{(b)}\in\{1,\dots,N_k\}$ for each group $k=K_0+1,\dots,K$, so that 
$\{Z_{K_0+1,i_{K_0+1}^{(b)}},\dots,Z_{K,i_K^{(b)}}\}$ is the subsampled calibration data set for the $b$th run of Subsampling Once.
To aggregate the $B$ many runs together,
\citet{dunn2022distribution} reformulate the Subsampling Once procedure in terms of a p-value,
rather than a quantile, and then average the p-values. 
To do this, define the p-value for the $b$th run of Subsampling Once as
\[P^{(b)}(z) =\frac{1 + \sum_{k=K_0+1}^K \one{s(Z_{k,i_k^{(b)}}) \geq s(z)}}{K_1+1}.\]
Then it holds that the
prediction interval for the $b$th run of Subsampling Once can equivalently be defined as $\Ch^{(b)}(X_{\test}) = \{y : P^{(b)}(X_{\test},y) >\alpha\}$, the set of values $y$ whose p-value is \emph{not} sufficiently
small to be rejected at level $\alpha$.

With this equivalent formulation in place, we are now ready to define the aggregation: 
 the prediction interval for Repeated Subsampling is given by
\[\Ch(X_{\test}) = \{y : \bar{P}(X_{\test},y) >\alpha\}\textnormal{ where }\bar{P} =  \frac{1}{B}\sum_{b=1}^B P^{(b)}(z).\]
This averaging strategy avoids the loss of information incurred by taking only a single subsample; however, \citet{dunn2022distribution}'s
 finite-sample theory for the 
Repeated Subsampling method only ensures $1-2\alpha$, rather than $1-\alpha$, as the marginal
coverage level. 
(The factor of 2 arises from the fact that an average of valid p-values is, up to a factor of 2, a valid p-value \citep{vovk2020combining}.) 

Next, we will see that using our HCP framework leads to a stronger finite-sample guarantee.
Specifically, we now
derive an equivalent formulation
to better understand the relationship of this method to HCP.
By examining the definition of $\bar{P}$, we can verify that
Repeated Subsampling
 is equivalent to taking $\Ch(X_{\test}) =  \left\{y\in\mathcal{Y}:s(X_{\test},y) \leq T\right\}$ where the threshold $T$
is defined as
\[T =  Q_{1-\alpha}\left(\sum_{b=1}^B \sum_{k=K_0+1}^K \frac{1}{B(K_1+1)}\delta_{s(Z_{k,i^{(b)}_k})} + \frac{1}{K_1+1}\delta_{+\infty}\right).\]
If the number of repeats $B$ is extremely large, 
we would expect that, within group $k$, the selected index $i^{(b)}_k$ is equal to each $i=1,\dots,N_k$ in roughly equal proportions---that is,
\[\sum_{b=1}^B  \frac{1}{B}\delta_{s(Z_{k,i^{(b)}_k})}  \approx  \sum_{i=1}^{N_k}\frac{1}{N_k}\delta_{s(Z_{k,i})} .\]
Consequently, as $B\rightarrow\infty$ in the Repeated Subsampling
method,
we would expect that the threshold $T$ constructed by ``Repeated sampling''
is approximately equal to the threshold $T$ constructed by HCP~\eqref{eqn:define_HCP}, and so the resulting prediction intervals should be approximately the same.
In other words, for large $B$, we can view the Repeated Subsampling method as an approximation to HCP.

In fact, we can actually derive a stronger result: under hierarchical exchangeability, Repeated Subsampling
offers a theoretical guarantee of coverage at level $1-\alpha$, rather than the weaker $1-2\alpha$ result in existing work
(However, this new result  requires exchangeability of the group sizes $N_k$, which the $1-2\alpha$ coverage result does not). Details are deferred to Appendix.

\subsection{Simulations}\label{section:simulation_1}

We next present simulation results to illustrate the comparison between 
HCP and the methods proposed by \citet{dunn2022distribution}.\footnote{Code to reproduce this simulation is available at \url{https://github.com/rebeccawillett/Distribution-free-inference-with-hierarchical-data}.}

\paragraph{Score functions.} A common score for conformal prediction is the absolute value of the residual,
\[
s(x,y) = \big|y - \widehat{\mu}(x)\big|,\]
where $\widehat{\mu}$ is a model fitted on the first portion of the training data (i.e., groups $k=1,\dots,K_0$, with the calibration set held out)---that is,
$\widehat{\mu}(x)$ is an estimate of $\EEst{Y}{X=x}$.
This score leads to prediction intervals of the form $\Ch(X_{\test}) = \widehat{\mu}(X_{\test})\pm T$, for some
threshold $T$ selected using the calibration set. We will use this score for this first simulation.

\afterpage{
\begin{figure}[H]
\centering
\begin{subfigure}[t]{\textwidth}
    \centering
    \includegraphics[width=0.75\textwidth]{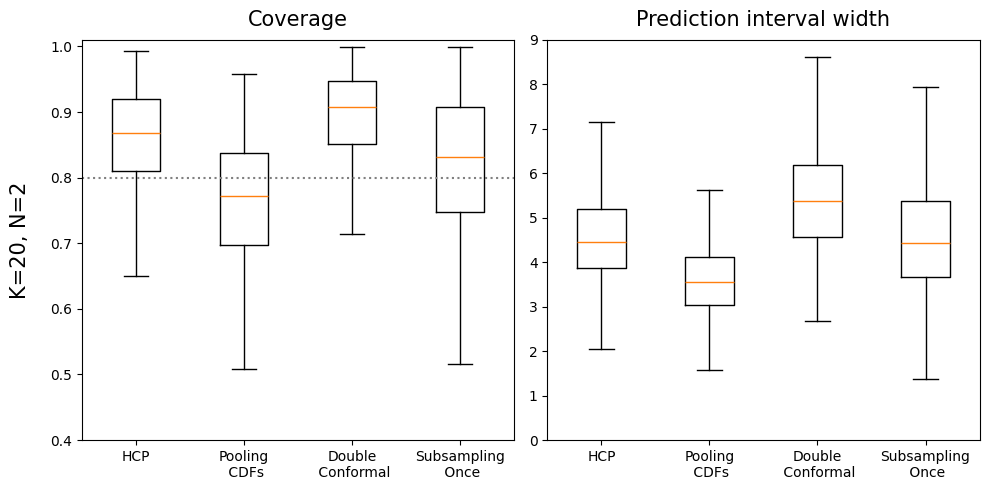}
\end{subfigure}

\vspace{0.5em}

\begin{subfigure}[t]{\textwidth}
    \centering
    \includegraphics[width=0.75\textwidth]{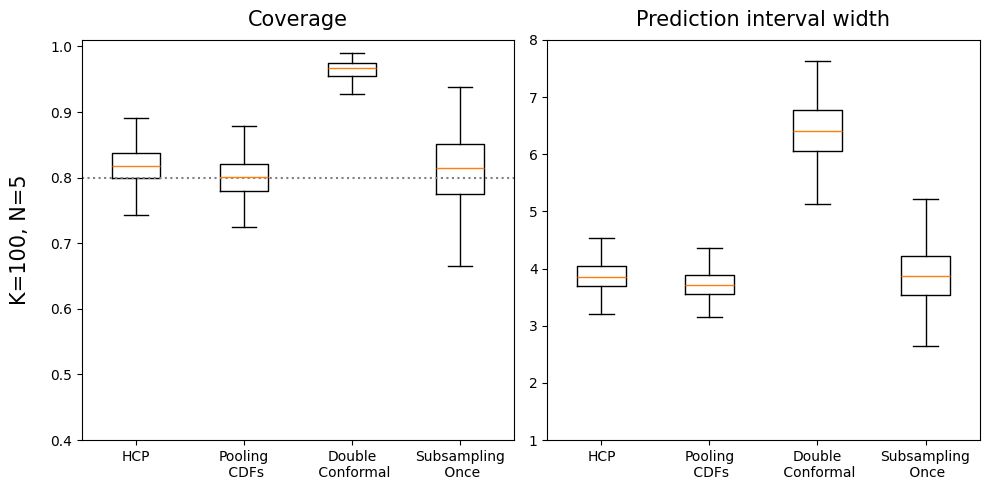}
\end{subfigure}

\caption{Coverage and prediction interval width of hierarchical conformal prediction (HCP), compared
with methods proposed by \citet{dunn2022distribution}, for $K=20,N=2$ (top row) and $K=100,N=5$ (bottom row). The plots show results over 500 independent trials,
where for each trial the result is averaged over a test set of 1000 independent points.}
\label{fig:HCP_simulation}
\end{figure}

\begin{table}[H]
\begin{center} {\small
\begin{tabular}{llll}
\hline
& & Coverage & Prediction interval width\\
\hline
\multirow{4}{*}{$K=20,N=2$}&HCP & 0.8562 (0.0037) & 4.5582 (0.0453)  \\
&Pooling CDFs \citep{dunn2022distribution} & 0.7620  (0.0045) & 3.6016 (0.0348)  \\
&Double Conformal \citep{dunn2022distribution} & 0.8937 (0.0032) & 5.4658 (0.0538) \\
&Subsampling Once \citep{dunn2022distribution} & 0.8151 (0.0051) & 4.5645 (0.0572)  \\\hline
\multirow{4}{*}{$K=100,N=5$}&HCP & 0.8177 (0.0013) & 3.8743 (0.0120)   \\
&Pooling CDFs \citep{dunn2022distribution} & 0.8002 (0.0013) & 3.7221 (0.0113)  \\
&Double Conformal \citep{dunn2022distribution} & 0.9645 (0.0006) & 6.4085 (0.0231)  \\
&Subsampling Once \citep{dunn2022distribution} & 0.8100 (0.0025) & 3.9017 (0.0222)  \\\hline
\end{tabular} }
\end{center}
\caption{Mean coverage and mean prediction interval widths of hierarchical conformal prediction (HCP), compared
with methods proposed by \citet{dunn2022distribution}, with standard errors in parentheses. Results are averaged over 500 independent trials.}\label{table:marginal_coverage}
\label{turns}
\end{table}}

\paragraph{Data and methods.}
The data for this simulation is generated as follows:
independently for each group $k=1,\dots,K$, we set $N_k \equiv N$ and draw
\begin{align*}
& \alpha_k, \beta_k \iidsim \mathcal{N}(0,1),\quad \left(\begin{array}{c}X_{k,1}\\\vdots\\X_{k,N}\end{array}\right) 
\ \,\Big\vert\, \  \alpha_k \sim \mathcal{N}\left(\left(\begin{array}{c}\alpha_k\\ \vdots\\ \alpha_k \end{array}\right), \mathbf{I}_N\right),\\
& \left(\begin{array}{c}Y_{k,1}\\\dots\\Y_{k,N}\end{array}\right)\ \,\Big\vert\, \ \left(\begin{array}{c}X_{k,1}\\\dots\\X_{k,N}\end{array}\right), \alpha_k,\beta_k \sim  N\left(\left(\begin{array}{c}\beta_k+\mu(X_{k,1})\\ \dots \\ \beta_k+ \mu(X_{k,N})\end{array}\right), \mathbf{I}_N\right),
\end{align*}
where the true conditional mean function is given by
\[\mu(x) = 1+x+0.1\cdot x^2,\] 
and $\mathbf{I}_N$ denotes the $N$-dimensional identity matrix. We can think of $\alpha_k$ and $\beta_k$ as random, group-specific effects.

We run the experiment for the following choices of $K$ (the number of groups) and $N$ (the number of observations within each group):
\[K=20,N=2, \quad K=100, N=5.\]
The target coverage level is set to be 80\%, i.e., $\alpha=0.2$.
The training portion of the data (i.e., groups $k=1,\dots,K_0$, for $K_0=K/2$) 
is used to produce an estimate
$\widehat{\mu}(x)$ of the conditional mean function via linear regression (specifically, we run least squares
on the data set $\{(X_{k,i},Y_{k,i}) : k=1,\dots,K_0, \  i=1,\dots,N\}$, ignoring the hierarchical structure of the data),
to define the score $s(x,y) = |y - \widehat\mu(x)|$ as described above.
We then use the calibration set (i.e., $K_1=K/2$ many groups $k=K_0+1,\dots,K$) to 
choose the threshold $T$ that defines the prediction interval, for all methods, as detailed in the sections above.
We compare HCP with \citet{dunn2022distribution}'s Pooling CDFs, Double Conformal, and Subsampling Once methods.
(We do not compare with Repeated Subsampling since, as explained in Section~\ref{sec:Dunn},
the Repeated Subsampling method with large $B$ can essentially be interpreted as an  approximation to 
HCP in this split conformal setting.)

\paragraph{Results.}
The results for all methods are shown in  Figure~\ref{fig:HCP_simulation}. Here, we generate the training set consisting of $K$ groups of size $N$,
together with a test set of $n_{\test}=1000$ independent draws $(X_{\test},Y_{\test})$, each drawn from the same distribution 
as the data points $(X_{k,i},Y_{k,i})$ in the training set.
For each of the four methods, we calculate the mean coverage rate,
\begin{equation}\label{eqn:plot_coverage}\frac{1}{1000}\sum_{i=1}^{1000}\one{Y_{\test,i}\in\Ch(X_{\test,i})},\end{equation}
and mean prediction interval width,
\begin{equation}\label{eqn:plot_width}\frac{1}{1000}\sum_{i=1}^{1000}\textnormal{length}(\Ch(X_{\test,i})).\end{equation}
The box plots in Figure~\ref{fig:HCP_simulation} show the distribution of these two measures of performance,
across 500 independent trials.
The overall average coverage for each method is also shown in Table~\ref{table:marginal_coverage}.

Comparing the results for the various methods, we see that the empirical results confirm 
our expectations as described above. In particular, among \citet{dunn2022distribution}'s methods, we see 
that Pooling CDFs works quite well but
shows some  undercoverage for small $K$ (as expected, since we have seen in Proposition~\ref{prop:pooling_cdfs}
that it is equivalent to HCP with a slightly higher value of $\alpha$);
that Double Conformal tends to provide over-conservative prediction sets (with coverage level $>1-\alpha$, and consequently,
substantially wider prediction intervals); and that Subsampling Once provides the correct coverage level but suffers from higher
variability due to discarding much of the data---to be specific, the method's performance on the test set can change dramatically when we resample
the training and calibration data, since the output of the method relies on a substantially reduced calibration set size for each trial.
HCP is free from these issues, and tends to show good empirical performance
 while also enjoying a theoretical finite-sample guarantee.

\section{Distribution-free prediction with repeated measurements}\label{sec:HCP2}

We now return to the setting of repeated measurements, introduced earlier in Section~\ref{sec:intro_setting}.
The training data consisting of i.i.d.\ samples of the features $X$ accompanied by repeated i.i.d.\ draws of the response $Y$,
\begin{equation}\label{eqn:model_repeated}\begin{cases} X_k\sim P_X, \\ N_k\mid X_k \sim P_{N|X}, \\ Y_{k,1},\dots,Y_{k,N_k}\mid (X_k,N_k)\iidsim P_{Y|X},\end{cases}\end{equation}
independently for each $k=1,\dots,K$,
and the test point is a new independent draw from the $(X,Y)$ distribution,
\begin{equation}\label{eqn:model_repeated_test_point}X_{\test}\sim P_X, \ Y_{\test}\mid X_{\test} \sim P_{Y|X}.\end{equation}
(This construction generalizes the model introduced earlier in~\eqref{eqn:intro_repeated}, because here we allow for potentially different numbers of repeated measurements for each $k$.)

Equivalently, we can draw from the model~\eqref{eqn:model_repeated} $K+1$ times, with $k=1,\dots,K$ denoting training data, 
and with test
point $(X_{\test},Y_{\test})=(X_{K+1},Y_{K+1,1})$ (without observing $Y_{K+1,2},\dots,Y_{K+1,N_{K+1}}$).
This model satisfies the definition of hierarchical exchangeability (Definition~\ref{def:hier_exch}), and therefore, the HCP
method defined in~\eqref{eqn:define_HCP} is guaranteed to offer marginal coverage at level $1-\alpha$, as in Theorem~\ref{thm:HCP}. 
Since the training data is not i.i.d.\ (and not exchangeable) due to the repeated measurements, this result is again novel relative to existing
split conformal methodology.

Thus far, we have only shown that HCP allows us to restore marginal predictive coverage \emph{despite} the repeated
measurements---that is, we can provide the same marginal guarantee for the repeated measurement setting,
as was already obtained by split conformal for the i.i.d.\ data setting. 
On the other hand, the presence of repeated measurements actually carries valuable
information---for instance, we can directly estimate the conditional variance of $Y|X$ at each $X=X_k$ for which we have multiple
draws of $Y$. This naturally leads us to ask whether we can use this special structure
to \emph{enhance} the types of guarantees it is possible to achieve---that is, can we expect a stronger inference result that achieves guarantees beyond the usual marginal coverage guarantee?
This is the question that we address next.

\subsection{Toward inference with conditional coverage guarantees}

In the setting of i.i.d.\ data (i.e., without repeated measurements), 
when the features $X$ have a continuous or nonatomic distribution, \citet{vovk2005algorithmic,lei2014distribution,barber2021limits} establish
that conditionally-valid predictive inference (that is, guarantees of the form $\PPst{Y_{\test} \in \Ch(X_{\test})}{X_{\test}}\geq 1-\alpha$) are impossible
to achieve with finite-width prediction intervals. 
However, access to repeated measurements may provide additional information that allows us to move beyond this impossibility
result. 
In this section, we show that access to repeated measurements allows us to go beyond the marginal coverage guarantee and enables control of the conditional coverage.

Consider the miscoverage rate conditional on all the observations we have---both the observed training and calibration data, $\mathcal{D} = \{(X_k,Y_{k,1},\dots,Y_{k,N_k})\}_{k=1,\dots,K}$, and the test feature vector, $X_{\test}$. Specifically, define
\[\alpha_{\mathcal{D}}(x) = \PPst{Y_{\test} \notin \Ch(X_{\test})}{\mathcal{D};X_{\test}=x},\]
where the probability is taken with respect to the distribution of the test response $Y_{\test}$ (which is drawn from $P_{Y|X}$ at $X_{\test}=x$).\footnote{Note
 that the random variable $\alpha_{\mathcal{D}}(X_{\test})$ is therefore the coverage conditional on both the observed
 training and calibration data $\mathcal{D}$, and 
test feature $X_{\test}$, meaning that we are now examining an even stronger form of conditional coverage. However,
it is well known that conditioning on $\mathcal{D}$ is straightforward for split conformal prediction \citep{vovk2012conditional},
that is, the primary challenge in bounding $\alpha_{\mathcal{D}}(X_{\test})$ comes from conditioning on $X_{\test}$.}
 The marginal coverage guarantee
\[\PP{Y_{\test} \notin \Ch(X_{\test})} \leq \alpha\]
can then equivalently be expressed as
\begin{equation}\label{eqn:guarantee_marginal}
\EE{\alpha_{\mathcal{D}}(X_{\test})} \leq \alpha,
\end{equation}
by the tower law.

For the goal of controlling conditional miscoverage, an ideal target would be
\begin{equation}\label{eqn:guarantee_ideal}
\alpha_{\mathcal{D}}(X_{\test}) \leq \alpha \text{ almost surely.}
\end{equation}
This condition, if achieved, would  ensure that the resulting prediction set provides a good coverage for any value of new input $X_{\test}$, rather than only on average over $X_{\test}$. However,
even in the setting of repeated measurements, if $X$ is continuously distributed (or more generally, nonatomic, i.e., $P_X(x) =0$ for all $x\in\mathcal{X}$), the guarantee~\eqref{eqn:guarantee_ideal} is not achievable by any nontrivial (finite-width) prediction interval. 

Consequently, we would like to find a coverage target that is stronger than the basic marginal guarantee~\eqref{eqn:guarantee_marginal}
but weaker than the unachievable conditional guarantee~\eqref{eqn:guarantee_ideal}.
As a compromise between the two,
we consider the following guarantee, which we  refer to as \emph{second-moment coverage}:
\begin{equation}\label{eqn:guarantee_stronger}
\EE{\alpha_{\mathcal{D}}(X_{\test})^2} \leq \alpha^2.
\end{equation}
This condition is a compromise between the previous two: second-moment coverage~\eqref{eqn:guarantee_stronger}
is strictly weaker than the (unattainable) conditional coverage guarantee~\eqref{eqn:guarantee_ideal},
and strictly stronger than the marginal coverage guarantee~\eqref{eqn:guarantee_marginal}.

Another way of understanding the guarantee~\eqref{eqn:guarantee_stronger} as a mechanism for controlling conditional miscoverage is to look at the tail probability. By Markov's inequality, for any constant $c > 1$, the marginal coverage guarantee leads to
\[\PP{\alpha_{\mathcal{D}}(X_{\test}) > c\alpha} \leq \frac{\EE{\alpha_{\mathcal{D}}(X_{\test})}}{c\alpha} = \frac{1}{c},\]
while the stronger guarantee~\eqref{eqn:guarantee_stronger} leads to a stricter bound of $1/c^2$ since it implies
\[\PP{\alpha_{\mathcal{D}}(X_{\test}) > c\alpha} = \PP{\alpha_{\mathcal{D}}(X_{\test})^2 > c^2\alpha^2} \leq \frac{\EE{\alpha_{\mathcal{D}}(X_{\test})^2}}{c^2\alpha^2} = \frac{1}{c^2},\]
Hence, we can expect more uniformly small conditional miscoverage rate from the second-moment coverage guarantee.

\subsection{Proposed method for second-moment coverage: \texorpdfstring{HCP$^2$}{}}
Now we discuss how the stronger guarantee~\eqref{eqn:guarantee_stronger} can be achieved in the distribution-free sense,
in the setting of repeated measurements. 
We propose a modification of our HCP method, which we call HCP$^2$---this name refers to the fact that the
 new goal is second-moment coverage.
(As before, here we aim to extend the split conformal method to address this goal;
extensions of other methods---namely, full conformal and jackknife+---are deferred to the Appendix.)

To begin, let $K_1^{\geq 2} = \sum_{k=K_0+1}^K \one{N_k\geq 2}$, the number of feature obesrvations in the calibration set
for which we have at least two repeated measurements.
\begin{multline}\label{eqn:define_HCP2}\textnormal{HCP$^2$ method: \quad }\Ch(X_{\test})= \{y\in\mathcal{Y} : s(X_{\test},y) \leq T\}, \textnormal{ where }\\
T =Q_{1-\alpha^2}\left(\sum_{\substack{k=K_0+1,\dots,K\\ N_k\geq 2}}\sum_{i=1,\dots,N_k}\frac{N_k - i}{(K_1^{\geq 2} +1){N_k\choose 2}}\cdot\delta_{s_{k,(i)}}+\frac{1}{K_1^{\geq 2}+1}\cdot\delta_{+\infty}\right),\end{multline}
where $s_{k,(1)}\leq \dots \leq s_{k,(N_k)}$ are the order statistics of the scores $s(Z_{k,1}),\dots,s(Z_{k,N_k})$ in the $k$th calibration
group.

To understand the intuition behind this construction, we can first verify that the threshold $T$ in this definition
can equivalently be written as
\[T =Q_{1-\alpha^2}\left(\sum_{\substack{k=K_0+1,\dots,K\\ N_k\geq 2}}\sum_{1 \leq i < i' \leq N_k}\frac{1}{(K_1^{\geq 2} +1){N_k\choose 2}}\cdot\delta_{\min\{s(Z_{k,i}),s(Z_{k,i'})\}}+\frac{1}{K_1^{\geq 2}+1}\cdot\delta_{+\infty}\right).\]
With this equivalent definition in place, we can now explain 
the main idea behind the construction of the HCP$^2$ method.
For intuition, suppose $Y\mid X$ has a continuous distribution.
For each group $k$, since the $Y_{k,i}$'s are i.i.d.\ conditional on $X=X_k$,
if $q_{1-a}$ is equal to the $(1-a)$-quantile of the distribution of $s(X,Y)$ conditional on $X=X_k$,
then for constructing HCP (and proving marginal coverage), we observe that $\PP{s(X_k,Y_{k,i})> q_{1-a}}=a$
by construction. 
But, if we have at least two measurements $1\leq i<i'\leq N_k$ at this particular $k$,
then we can also see that
$\PP{\min\{s(X_k,Y_{k,i}),s(X_k,Y_{k,i'})\}> q_{1-a}} = a^2$. In other words, by examining the pairwise minimums
across the scores at each value of $X$, we can learn about $\alpha_{\mathcal{D}}(X)^2$, rather than learning only about $\alpha_{\mathcal{D}}(X)$.

The following theorem verifies that the HCP$^2$ method achieves second-moment coverage---with the caveat
that the coverage guarantee is with respect to a reweighted distribution on $X_{\test}$.
In particular,
let $p_{\geq 2} = \PP{N\geq 2}>0$ denote the probability of at least two repeated measurements when drawing a new data point,
and let $P_X^{\geq 2}$ denote the conditional distribution of $X$,
when conditioning on the event $N\geq 2$ under the joint model $(X,N)\sim P_X\times P_{N|X}$.
\begin{theorem}[Second-moment coverage for HCP$^2$]\label{thm:HCP2}
Assume the training data is drawn from the i.i.d.\ model with repeated measurements~\eqref{eqn:model_repeated},
independently for $k=1,\dots,K$, and that the test point $(X_{\test},Y_{\test})$ is drawn
independently from the model~\eqref{eqn:model_repeated_test_point}. Then
\[\Ep{P_X^{\geq 2}}{\alpha_{\mathcal{D}}(X_{\test})^2} \leq \alpha^2.\]
Moreover, if all scores $s(Z_{k,i})$ are distinct almost surely, it additionally holds that
\[\Ep{P_X^{\geq 2}}{\alpha_{\mathcal{D}}(X_{\test})^2} \geq \alpha^2 - \frac{2}{(K_1+1)p_{\geq 2}}.\]
\end{theorem}
\noindent
In particular, if the number of repeated measurements $N$ is independent from $X$ (i.e., $N_k\independent X_k$ for each $k$---for example,
if the number of repeats is simply given by some constant greater than or equal to two),
then $P_X^{\geq 2}$ is simply equal to the original marginal distribution $P_X$, 
and we have established a second-moment coverage guarantee $\EE{\alpha_{\mathcal{D}}(X_{\test})^2}\leq \alpha^2$, 
as in the initial goal~\eqref{eqn:guarantee_stronger}.

\subsection{Remarks}

\subsubsection{A trivial way to bound the second moment}\label{sec:trivial}
Even in a setting where measurements are not repeated,
it is possible to achieve a second-moment bound in a trivial way: since $\alpha_{\mathcal{D}}(X_{\test})\leq 1$ holds always,
it's trivially true that $\EE{\alpha_{\mathcal{D}}(X_{\test})^2}\leq \EE{\alpha_{\mathcal{D}}(X_{\test})}$.
However, this type of bound is not meaningful. For example, suppose we set $\alpha = 0.1$, aiming for 90\% coverage.
If we construct prediction intervals with 90\% marginal coverage, then we achieve $ \EE{\alpha_{\mathcal{D}}(X_{\test})} \leq 0.1$.
If we instead construct prediction intervals with 99\% marginal coverage (e.g., run split conformal, or HCP
if appropriate, with $\alpha = 0.01$ in place of $\alpha = 0.1$),
then we  achieve $ \EE{\alpha_{\mathcal{D}}(X_{\test})} \leq 0.01$ and consequently, the second-moment bound 
$\EE{\alpha_{\mathcal{D}}(X_{\test})^2}\leq 0.1^2$  also holds. 

However, this is not a method we would want to use in practice because it would 
mean that we are substantially inflating the size of our prediction intervals.
Instead, we would like to verify whether our \emph{original} prediction intervals---constructed with 90\% marginal coverage---achieve
the second-moment bound as well, because they exhibit good conditional coverage properties. 
When repeated measurements are available, this is the benefit offered by
 HCP$^2$. As we  see below, in a setting where the prediction intervals
constructed by HCP already show good conditional (as well as marginal) coverage, 
the intervals constructed by HCP$^2$ need not be much wider.
In addition, to further illustrate this point, in Appendix, we  show empirically that HCP$^2$ can improve
substantially over the trivial solution given by running HCP with $\alpha^2$ in place of $\alpha$.

\subsubsection{Beyond the second moment?}
If the groups (or at least, a substantial fraction of groups) have at least two observations, 
we have seen that we can move from marginal coverage to second-moment coverage---that is, move from bounding
$\EE{\alpha_{\mathcal{D}}(X_{\test})}$, to bounding $\EE{\alpha_{\mathcal{D}}(X_{\test})^2}$. 
In principle, it is possible to push this further: if our data distribution includes groups with $\geq \ell$ observations in the group,
it is possible to define a method for $\ell$-th moment coverage, i.e., bounding $\EE{\alpha_{\mathcal{D}}(X_{\test})^\ell}$.
The construction would be analogous to HCP$^2$---instead of considering pairs of scores within each group (i.e., $\min\{s(Z_{k,i}),s(Z_{k,i'})\}$), we would consider
$\ell$-tuples. However, we do not view this extension as practical (unless the number of groups $K$ is massive), because
we would need to compute the threshold for the prediction interval via a $(1-\alpha^\ell)$-quantile; since in practice
we are generally interested in small values of $\alpha$ (e.g., $\alpha = 0.1$), this would likely lead to intervals
that are not stable (and often, not even finite).

\subsection{Simulations}\label{section:simulation}
We next show simulation results to illustrate the performance of HCP$^2$, as compared to HCP,
for the setting of repeated measurements.\footnote{Code to reproduce this simulation is available at \url{https://github.com/rebeccawillett/Distribution-free-inference-with-hierarchical-data}.}

\paragraph{Score functions.} We  consider two different score functions in our simulations.
First, as for our earlier simulation, we will consider the simple residual score
\begin{equation}\label{eqn:score1}
s(x,y) = \big|y - \widehat{\mu}(x)\big|,\end{equation}
where $\widehat\mu(x)$ is an estimate of $\EEst{Y}{X=x}$, fitted on the first portion of the training data (i.e., $k=1,\dots,K_0$).
As described earlier, this score leads to prediction intervals of the form  $\Ch(X_{\test}) = \widehat{\mu}(X_{\test})\pm T$ for some
threshold $T$, and therefore may not be ideally suited for data with nonconstant variance in the noise.
This motivates a variant of the above score,
given by a rescaled residual,
\begin{equation}\label{eqn:score2}
s(x,y) = \frac{\big|y - \widehat{\mu}(x)\big|}{\widehat{\sigma}(x)},\end{equation}
where $\widehat{\sigma}(x)^2$ is an estimate of $\textnormal{Var}(Y|X=x)$, again
fitted on the first portion of the training data (i.e., $k=1,\dots,K_0$).
This score leads to prediction intervals of the form $\Ch(X_{\test}) = \widehat{\mu}(X_{\test})\pm T\cdot\widehat{\sigma}(X_{\test})$,
again for some
threshold $T$ selected using the calibration set.
The rescaled residual score
has the potential to adapt to local variance and therefore, to exhibit better conditional coverage properties empirically (see e.g.,~\citet{lei2018distribution} for more discussions on the use of this nonconformity score).
In this simulation we will use each of the above scores.

\paragraph{Data and methods.}
The data for this simulation is generated as follows:
for each group $k=1,\dots,K$, we set $N_k \equiv 2$ and draw
\begin{align*}
&X_k \iidsim \text{Unif}[0,5],\\
&Y_{k,1},Y_{k,2}\mid X_k \iidsim \mathcal{N}(\mu(X_k), \sigma(X_k)^2),
\end{align*}
where the true conditional mean is given by
\[\mu(x) = 1+x+0.1\cdot x^2 ,\]
 and the true conditional variance is defined in two different ways:
\begin{align*}
&\text{Setting 1 (constant variance)} : \sigma(x) \equiv 2, \\
&\text{Setting 2 (nonconstant variance)} : \sigma(x) = \One{x < 3} + (1+4(x-3)^4)\cdot \One{3 \leq x < 4} + 5\cdot \One{x \geq 4}.
\end{align*}
For both settings, we generate data for $K=1000$ many independent groups, with $N=2$ observations per group.
The test point $X_{\test}$ is then drawn independently from the same marginal distribution, $X_{\test}\sim \text{Unif}[0,5]$. 

The two settings are illustrated in Figure~\ref{fig:scatter}. Setting 1 reflects an easier setting, where the constant variance of the noise means that
the scores~\eqref{eqn:score1} and~\eqref{eqn:score2} each lead to prediction intervals that fit well to the trends in the data as long as $\widehat{\mu}$ estimates the true conditional 
mean $\mu$ fairly accurately. Setting 2, on the other hand, reflects the case where we have different degrees of prediction difficulty at different values of the features $X$
due to highly nonconstant variance in the noise,
and in particular, the simple score~\eqref{eqn:score1} is not well-suited to this setting (since it does not allow for prediction intervals to have widths varying with $X$),
while the rescaled residual score~\eqref{eqn:score2} has the potential to yield a much better fit as long as both $\widehat{\mu}$ and $\widehat{\sigma}$ are accurate estimates.

The training data is split into $K_0=K/2$ many groups for training
and $K_1 = K/2$ many groups for calibration. On the training data, $\widehat{\mu}$ and $\widehat{\sigma}$  are fitted 
via kernel regression with a box kernel $K_h(x) = \frac{1}{2h}\cdot\One{|x| < h}$, with bandwidth $h = 0.5$.

\begin{figure}[ht]
\begin{center}
\includegraphics[width=0.9\textwidth]{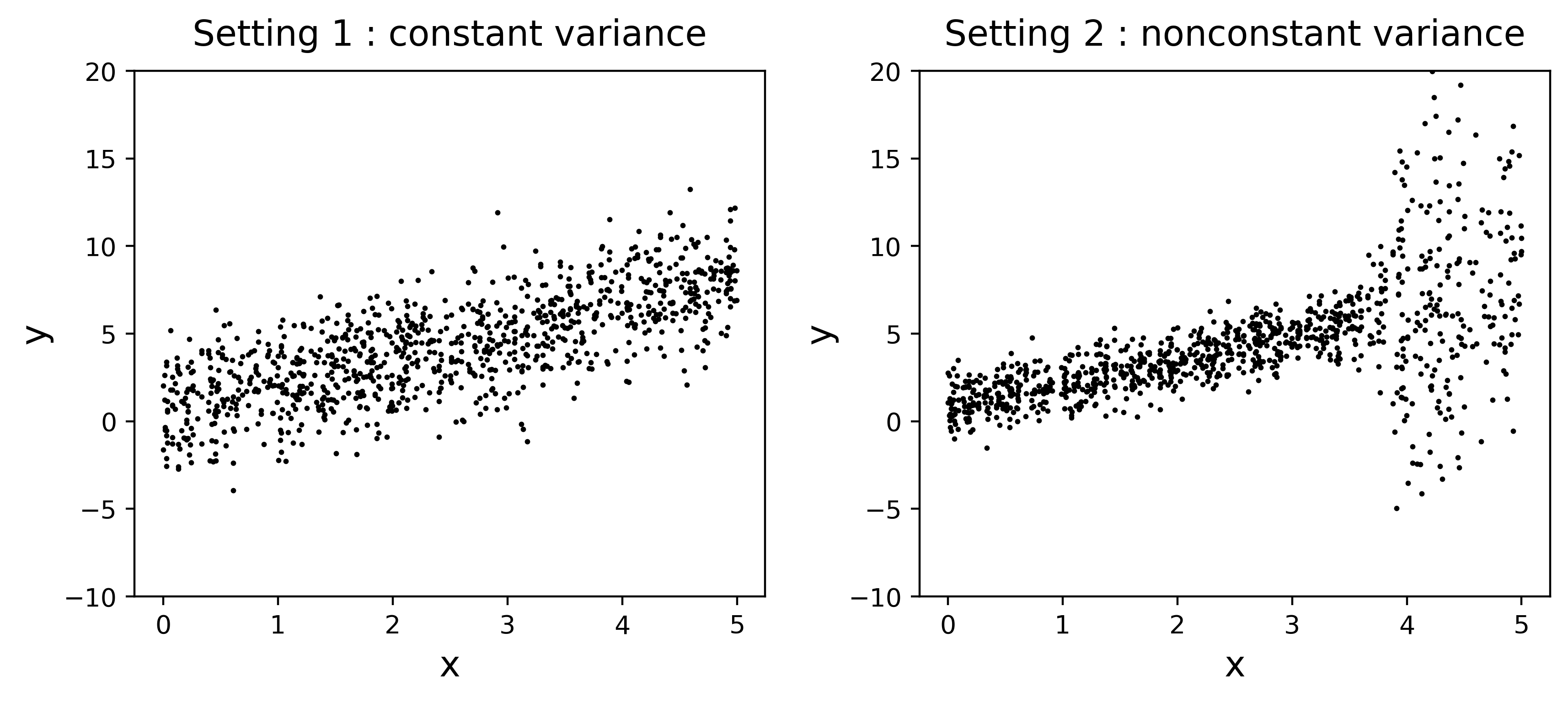}
\caption{Scatter plot of data sets from one trial of Setting 1 (constant variance case) and one trial of Setting 2 (nonconstant variance case).}
\label{fig:scatter}
\end{center}
\end{figure}

\paragraph{Results.}
We run 500 independent trials of the simulation. For each trial, each of the two settings, and each method (HCP and HCP$^2$), 
after generating the training data and the test point $X_{\test}$ we construct a prediction set $\Ch(X_{\test})$,
and then compute the conditional miscoverage rate $\alpha_{\mathcal{D}}(X_{\test})$ (which can be computed exactly,
since the distribution of $Y_{\test}\mid X_{\test}$ is known by construction) and the width of the prediction set $\text{length}(\Ch(X_{\test}))$.

The results for the residual score $s(x,y) = |y - \widehat\mu(x)|$  are shown in Figure~\ref{fig:sim_hist_residual}.
We see that, for Setting 1 (constant variance), where this score is a good match for the shape of the conditional distribution of $Y\mid X$,
both HCP and HCP$^2$ show conditional miscoverage rates $\alpha_{\mathcal{D}}(X_{\test})$ that cluster closely around 
the target level $\alpha =0.2$, and consequently, the prediction interval widths are similar for the two methods. 
For Setting 2 (nonconstant variance), on the other hand, this score is a poor match for the distribution, and while HCP achieves
an average coverage of 80\%, the distribution of $\alpha_{\mathcal{D}}(X_{\test})$ has extremely large variance. 
It is clear that HCP does not satisfy any sort of conditional coverage or second-moment coverage property,
and consequently HCP$^2$ is forced to produce substantially wider prediction intervals, illustrating the tradeoff between
width of the interval and the coverage properties it achieves.

\afterpage{
\begin{figure}[H]
\begin{center}
\includegraphics[width=0.73\textwidth]{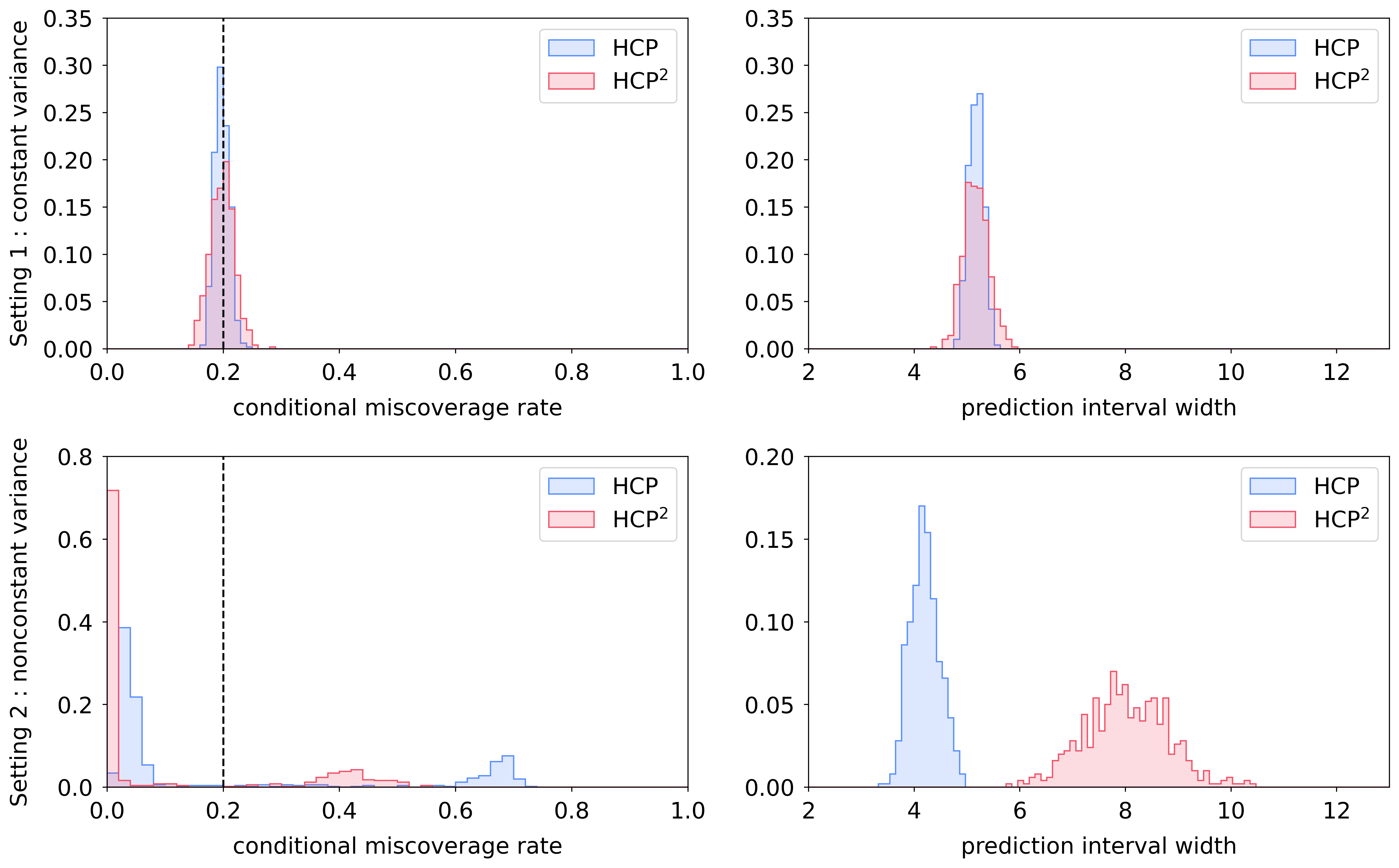}
\caption{Results for HCP and HCP$^2$ over 500 independent trials, with score $s(x,y) = |y - \widehat{\mu}(x)|$.}
\label{fig:sim_hist_residual}
\end{center}
\end{figure}\vspace{-.2in}

\begin{figure}[H]
\begin{center}
\includegraphics[width=0.73\textwidth]{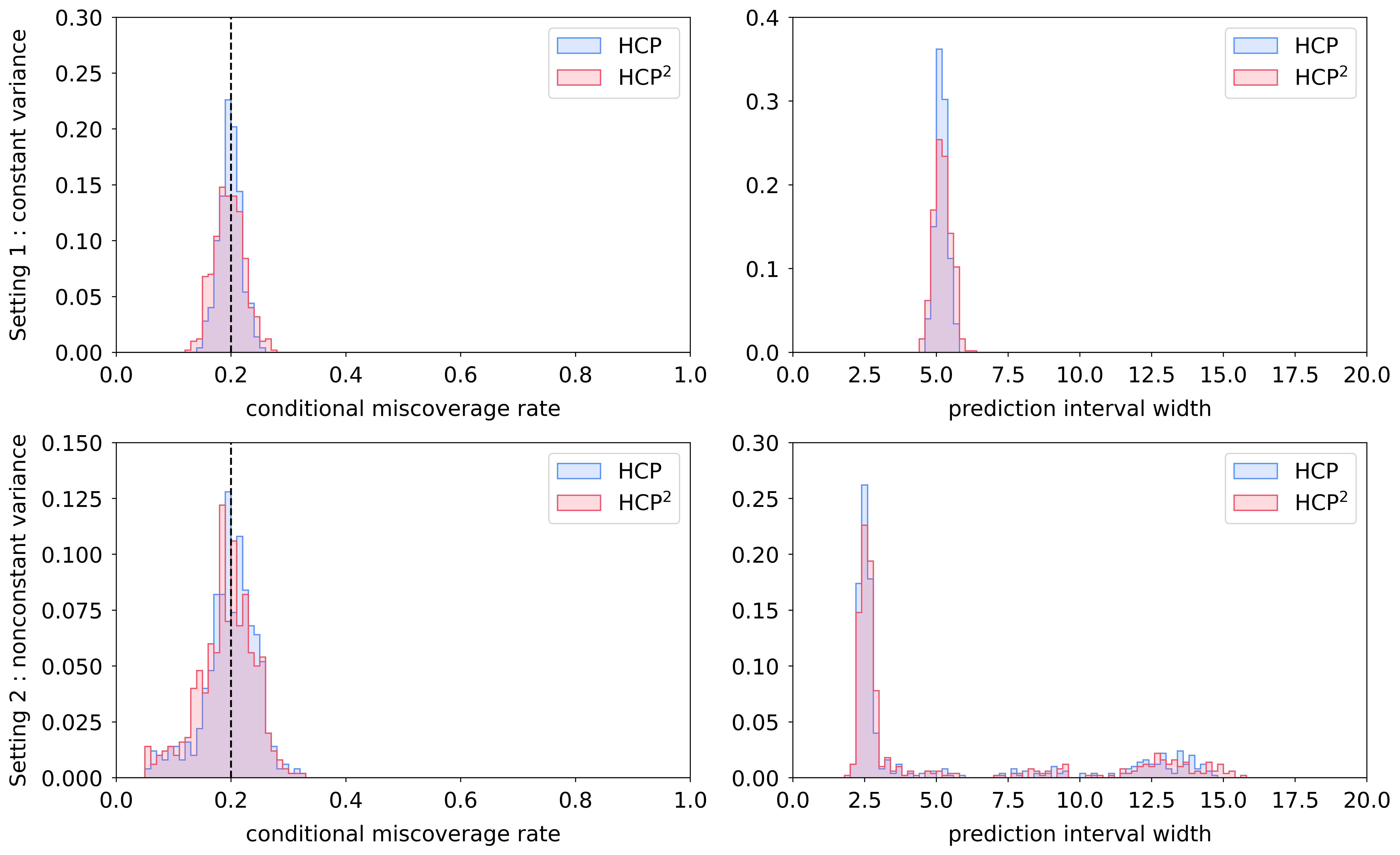}
\caption{Results for HCP and HCP$^2$ over 500 independent trials, with score $s(x,y) = |y - \widehat{\mu}(x)|/\widehat{\sigma}(x)$.}
\label{fig:sim_hist_standardized}
\end{center}
\end{figure}\vspace{-.2in}

\begin{table}[h]
\begin{center} {\small
\begin{tabular}{cc|cc|cc}
\hline
&&\multicolumn{2}{c|}{$s(x,y) = |y - \widehat{\mu}(x)|$}&\multicolumn{2}{c}{$s(x,y) = |y - \widehat{\mu}(x)|/\widehat\sigma(x)$}\\
&&$\EE{\alpha_{\mathcal{D}}(X_{\test})}$ & $\EE{\alpha_{\mathcal{D}}(X_{\test})^2}$ &$\EE{\alpha_{\mathcal{D}}(X_{\test})}$ & $\EE{\alpha_{\mathcal{D}}(X_{\test})^2}$\\
\hline
\multirow{2}{*}{Setting 1}& HCP &  0.1979 (0.0006)  & 0.0393 (0.0002)  &  0.1985 (0.0009)  & 0.0398 (0.0003) \\
& HCP$^2$ & 0.1983 (0.0009)  & 0.0398 (0.0004)  & 0.1965 (0.0012)  & 0.0393 (0.0005) \\
\hline
\multirow{2}{*}{Setting 2} & HCP & 0.2017 (0.0118)  & 0.1108 (0.0084)  & 0.1982 (0.0020)  & 0.0414 (0.0008) \\
& HCP$^2$ & 0.1011 (0.0079)  &  0.0413 (0.0034)  & 0.1917 (0.0022)  & 0.0391 (0.0008) \\
\hline
\end{tabular} }
\end{center}
\caption{Mean miscoverage rate (target: $\alpha = 0.2$) and mean second-moment miscoverage rate (target: $\alpha^2 = 0.04$), for HCP and HCP$^2$, with standard errors in parentheses. Results are averaged over 500 independent trials.}
\label{tab:HCP_HCP2}
\end{table}}

In contrast, the results
for the rescaled residual score $s(x,y) = |y-\widehat\mu(x)|/\widehat\sigma(x)$ are shown in Figure~\ref{fig:sim_hist_standardized}.
This choice of score is a good match for the data distribution in both Setting 1 and Setting 2, since it accommodates nonconstant
variance. Consequently, for both HCP and HCP$^2$, we see that the empirical distribution of $\alpha_{\mathcal{D}}(X_{\test})$ 
clusters around the target level $\alpha = 0.2$, and the two methods produce similar prediction interval widths.
Thus we see that, when the score function is chosen well, the stronger coverage guarantee offered by HCP$^2$ is essentially ``free''
since we can achieve it without substantially increasing the width of the prediction intervals.

\section{Application to forecasting with dynamical system model}

Our study is inspired by the generation of ensembles of climate forecasts created by running a simulator over a collection of perturbed initial conditions. These so-called ``initial condition ensembles'' help decision-makers account for uncertainties in forecasts when developing climate change mitigation and response plans \citep{mankin2020value}. Similar ensembles are used for weather forecasting; for instance, the European Centre for Medium-Range Weather Forecasts (ECMWF) notes:
\begin{quote}
	An ensemble weather forecast is a set of forecasts that present the range of future weather possibilities. Multiple simulations are run, each with a slight variation of its initial conditions and with slightly perturbed weather models. These variations represent the inevitable uncertainty in the initial conditions and approximations in the models. \citep{ecmwf}
\end{quote}

\paragraph{Overview of experiment.}
In the context of this paper, we seek to take an initial condition and generate a prediction interval that contains all ensemble forecasts with high probability. 
Specifically, we apply HCP and HCP$^2$ to a dynamical system dataset generated from the Lorenz 96 (L96) model. For this experiment,\footnote{Code to reproduce this experiment is available at \url{https://github.com/rebeccawillett/Distribution-free-inference-with-hierarchical-data}.}  our task will be to \emph{predict the future state of the system} looking ahead from current time $t=T_0$ to future time $t=T$.
We will consider two different choices of $T$ for defining both an easier task (near-future prediction with a low value of $T$) and a harder task (prediction farther into the future, with a higher value of $T$).
We will see that, while HCP provides marginal coverage (as guaranteed by the theory) for both tasks, it empirically achieves the second-moment
coverage guarantee for only the easier task; consequently, HCP$^2$ is able to return intervals
that are only barely wider for the easier near-future prediction task, but for the harder task must return substantially wider intervals
in order to achieve a second-moment coverage guarantee.

\paragraph{A description of the L96 system}

Developed by \citet{lorenz1996predictability}, the L96 system is a prototype model for studying prediction and filtering of dynamical systems in climate science. 
The model is a system of coupled ordinary equations and defined as follows:

\begin{equation}\label{eqn:L96_model}
    \frac{\mathsf{d}u_{m}(t)}{\mathsf{d}t} = -u_{m-1}(t)\big(u_{m-2}(t) - u_{m+1}(t)\big) - u_{m}(t) + F,
\end{equation}
where $u_{m}(t)$ denotes the state of the system at $M$ different locations $m = 1, \dots ,M$.
In order for the system to be well-defined at the boundaries, we define $u_{M+1}(t)=u_{1}(t)$, $u_0(t) = u_M(t)$, $u_{-1}(t) = u_{M-1}(t)$; 
this ensures that the system is cyclic over the $M$ locations. 
The value $F \in \mathbb{R}$ denotes the external forcing, i.e., an effect of external factors; we use $F=10$, which corresponds to strong chaotic turbulence \citep{majda2005information}. A representative set of samples at discrete time points for the L96 system is shown in \cref{fig:L96illustration} for $M = 40$ locations from time $t=0.05$ to $t=5.00$. The figure illustrates correlations between neighboring channels at each time, reflecting how in temporal change in channel $m$'s amplitude depends on the amplitudes in channels $m$, $m-1$, $m-2$, and $m+1$ in \cref{eqn:L96_model}. This figure is for illstrative purposes only, since we use different values of $M$ and $T$ in our experiments described below.

\begin{figure}[h]
	\begin{center}
		\includegraphics[width=0.73\textwidth]{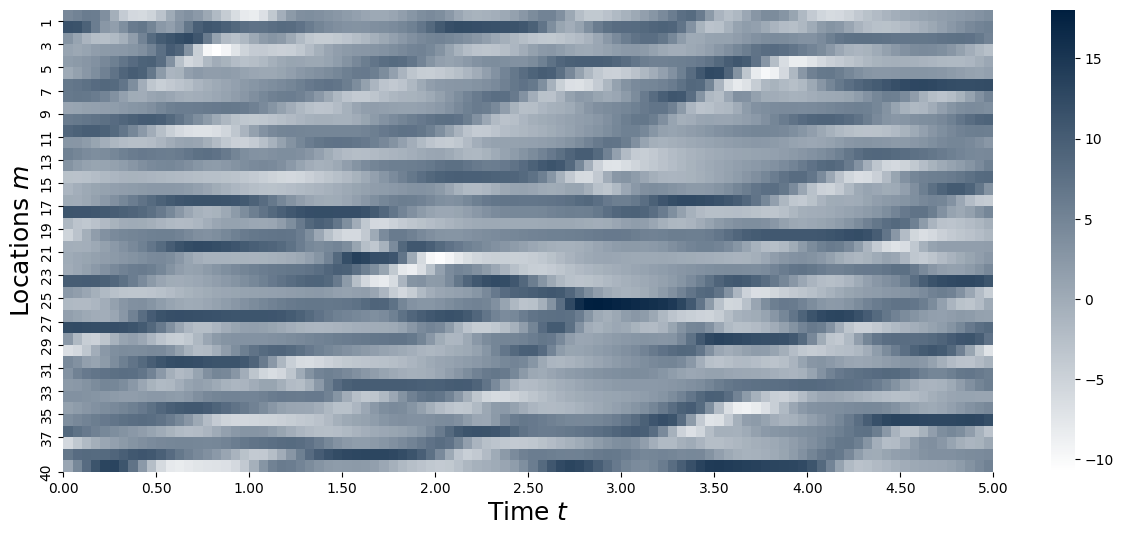}
		\caption{Illustration of the L96 dynamical system model: heatmap showing the states at $M=40$ locations across 80 time points between time 0 and 5.}
		\label{fig:L96illustration}
	\end{center}
\end{figure}

To specify our workflow, we  define a function $\texttt{L96}(\cdot;T)$ that runs the L96 system as defined in~\eqref{eqn:L96_model} for a time duration $T$ (more precisely, this function implements the Runge--Kutta approximation with time step $\mathsf{d}t = 0.05$, \citep{bank1987springer,solve_ivp}).
Thus, given starting conditions $u(0) = (u_1(0),\dots,u_M(0))$, the value of the system at time $T$, $u(T) = (u_1(T),\dots,u_M(T))$, is (approximately)
given by
\[u(T) = \texttt{L96}(u(0);T).\]

\paragraph{Generating the data.}

To generate data, we do the following: We let $K=800$ correspond to the number of independent groups of observations; in this setting, each group corresponds to a different initial states of the L96 system. 
Independently for each $k=1,\dots,K$,
we form feature vectors $X_k$ using the L96 system in order to provide more meaningful or realistic starting conditions:
\[X_k = \texttt{L96}\big(u; T_0\big) \textnormal{ where } u\sim \mathcal{N}(0,\mathbf{I}_m),\]
i.e., $X_k$ is the output of the L96 system run for  time duration $T_0$ (where we take $T_0=1$) when initialized with Gaussian noise at $t=0$.
This feature vector is $X_k = (X_{k1},\dots,X_{kM})\in\R^M$, with dimension $M=10$, which defines the initial conditions of the system.
To mimic an initial condition ensemble, we generate data as follows: For each initial condition $k$, we create an ensemble of $N_k=50$ perturbed initial conditions. For each $i=1,\dots,N_k$, start the L96 system initialized at a perturbation of $X_k$,
\[ (X_{k1} + r \eta^{k,i}_1,\dots,X_{kM} + r \eta^{k,i}_M)\textnormal{ where }\eta^{k,i}_m\iidsim \mathcal{N}(0,1),\]
with the scale of the noise set as $r = 0.01$.
We then run the L96 system~\eqref{eqn:L96_model} up until time $t=T = T_0+\tau$ to define the response:
\[Y_{ki} = \left[\texttt{L96}\Big( (X_{k1} + r \eta^{k,i}_1,\dots,X_{kM} + r \eta^{k,i}_M); T\Big) \right]_1,\]
i.e., 
the value of the system at the final time $T$, at its first location, $m=1$. Here, the term $\tau = T - T_0$ represents the forecast horizon, which reflects the difficulty of prediction. In the context of the weather and climate forecasting motivating application, our setting
is analogous to having imperfect knowledge of the current state of various weather variables ($X_k$) and observing an ensemble of predictions of precipitation at a
single spatial location in the future ($\{Y_{ki}\}_i$).

\paragraph{Implementation of the methods}

We consider two different choices of $\tau$, to define two different prediction tasks: a near-future prediction task, with $\tau=0.05$, and a 
more distant-future prediction task, with $\tau=0.5$.
For each of the two tasks, we generate training, calibration, and test sets as described above, where each data set has $K=800$ and $N_k\equiv 50$ (This can be interpreted as running a climate simulation for $N = 50$ different small perturbations of the true initial climate
state for each of $K = 800$ different true initial climate states). We run a simple linear neural network model on the training data set to produce estimators $\hat{\mu}(\cdot)$ and $\hat{\sigma}(\cdot)$ of the conditional mean and conditional standard deviation of $Y$ given $X$. We use the stochastic gradient descent optimizer with a learning rate of $0.0009$ and the mean squared error loss. Using the calibration data, we then implement the HCP and HCP$^2$ methods, with score function $s(x,y) = |y - \hat{\mu}(x)| / \hat{\sigma}(x)$.

\paragraph{Results.} 
We evaluate the performance of the two procedures with the test data. The results are summarized in Table~\ref{table:L96} and Figure~\ref{fig:L96}.

\begin{table}[h]
	\begin{center} {\small
			\begin{tabular}{lc|ccc}
				\hline
				&&$\EE{\alpha_{\mathcal{D}}(X_{\test})}$ & $\EE{\alpha_{\mathcal{D}}(X_{\test})^2}$ & Mean width\\
				\hline
				\multirow{2}{*}{$\tau = 0.05$}& HCP &  0.2006 (0.0041)  & 0.0546 (0.0029)  &  1.4875 (0.0044) \\
				& HCP$^2$ & 0.1748 (0.0039)  & 0.0437 (0.0026)  & 1.5753 (0.0047)  \\
				\hline
				\multirow{2}{*}{$\tau = 0.5$} & HCP & 0.2313 (0.0119)  & 0.1701 (0.0110)  & 10.4259 (0.1957) \\
				& HCP$^2$ & 0.0628 (0.0073)  &  0.0478 (0.0066)  & 22.0395 (0.4138)  \\
				\hline
		\end{tabular} }
	\end{center}
	\caption{Results for L96 data, for the HCP and HCP$^2$ methods. The table displays the estimated mean miscoverage rate (target: $\alpha = 0.2$), mean second-moment miscoverage rate (target: $\alpha^2 = 0.04$), and mean width of prediction intervals. These results are computed for a single trial of the experiment;  the mean miscoverage rate and mean prediction interval width are estimated using a test set (with standard errors shown in parentheses).}
	\label{table:L96}
\end{table}

\begin{figure}[h]
	\begin{center}
		\includegraphics[width=0.73\textwidth]{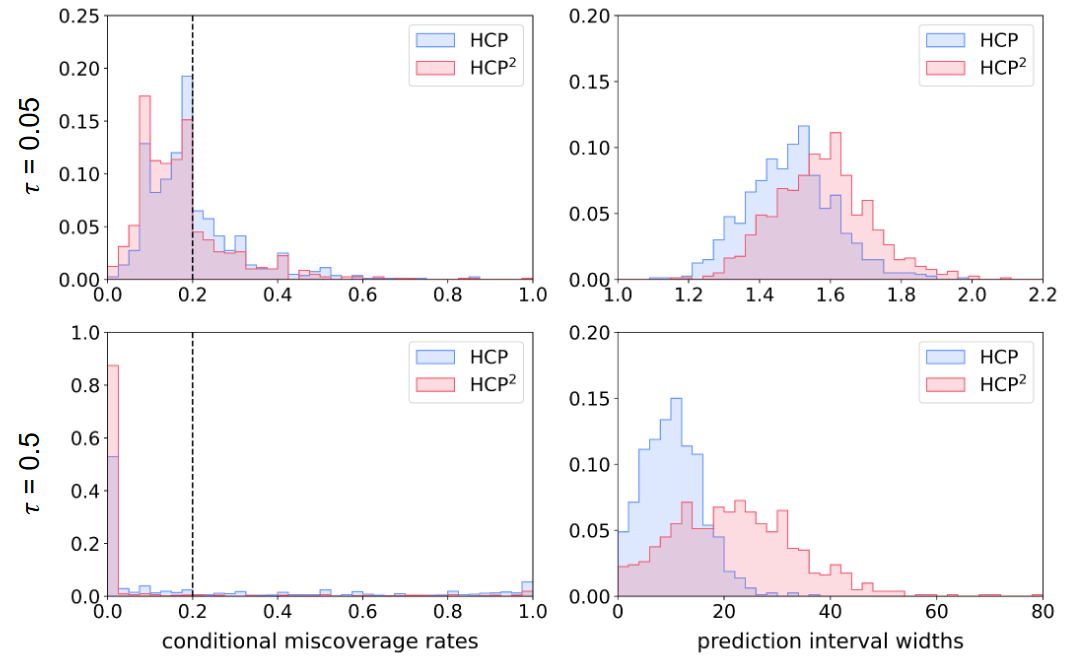}
		\caption{Results for L96 data: conditional miscoverage rates and the widths of the prediction sets from HCP and HCP$^2$.}
		\label{fig:L96}
	\end{center}
\end{figure}

The results show that HCP and HCP$^2$ provide prediction sets that satisfy their respective target guarantees. Since the dynamical system is chaotic, making accurate predictions becomes harder for larger $\tau$. For the task of predicting the relatively near future ($\tau=0.05$), 
the two procedures provide similar prediction sets. In particular, we see that HCP already achieves second-moment miscoverage near the target level $\alpha^2 =0.04$; consequently, HCP$^2$ is able to return a similar prediction interval as HCP, without being substantially more conservative. In contrast, for the task of predicting the far future ($\tau=0.5$), 
 while HCP successfully bounds the marginal miscoverage rate with relatively narrow prediction intervals, 
 it fails to achieve good conditional miscoverage control (indeed,
 HCP shows a high second-moment miscoverage rate---nearly as large as $\alpha = 0.2$, i.e., the worst-case scenario as described in Section~\ref{sec:trivial}). 
 As a result the HCP$^2$ method
needs to be much more conservative (wider prediction sets than HCP) in order to reduce second-moment miscoverage to the allowed level $\alpha^2 =0.04$. 

\section{Discussion}\label{sec:discussion}

In this work, we examine the problem of distribution-free inference for data sampled under a hierarchical structure, and 
propose hierarchical conformal prediction (HCP) to extend the conformal prediction framework to this setting.
We also examine the special case of repeated measurements (multiple draws of the response $Y$ for each sampled feature vector $X$),
and develop the HCP$^2$ method to achieve a stronger second-moment coverage guarantee, moving towards conditional coverage in this setting.

These results raise a number of open questions. In many statistical applications, the hierarchical structure of the sampling 
scheme is (at least partially) controlled by the analyst designing the study. In particular, this means that the analyst
can choose between, say, a large number of independent groups $K$ with a small number of measurements $N_k$ within each group,
or conversely a small number of groups with large numbers of repeats. 
Characterizing the pros and cons of this tradeoff is an important question to determine how study
design affects inference in this distribution-free setting.

Our results show that having even a small amount of repeated measurements can significantly expand the realm of what distribution-free inference can do (relatedly, recent work by \citet{lee2021distribution} examines a similar phenomenon when the distribution of $X$ is discrete). This raises a more general question: what reasonable additions can we make to the setting so that a more useful inference is possible in a distribution-free manner? We aim to explore this type of problem in our future work.

\subsection*{Acknowledgements}
Y.L., R.F.B., and R.W. were partially  supported by the National Science Foundation via grant DMS-2023109. 
Y.L. was additionally supported by National Institutes of Health via grant R01AG065276-01.
R.F.B. was additionally supported by the Office of Naval Research via grant N00014-20-1-2337.
R.W. was additionally supported the National Science Foundation via grant DMS-1930049 and by the Department of Energy via grant DE-AC02-06CH113575.

\bibliographystyle{plainnat}
\bibliography{bib}

\appendix

\section{Additional results for Repeated Subsampling}

\begin{proposition}\label{prop:repeated_subsampling}
Consider a bootstrapped calibration set, where for the $k$th
group, the data points within that group are given by
\begin{equation}\label{eqn:bootstrap_data set}\tilde{Z}^{\textnormal{boot}}_k = (Z_{k,i_k^{(1)}},\dots,Z_{k,i_k^{(B)}}),\end{equation}
where for each $k$, the indices $i_k^{(1)},\dots,i_k^{(B)}$ are sampled uniformly with replacement from $\{1,\dots,N_k\}$.
The Repeated Subsampling method of \citet{dunn2022distribution}  is equivalent to running HCP 
using the bootstrapped calibration set~\eqref{eqn:bootstrap_data set} (for $k=K_0+1,\dots,K$) in place of the original data set.
Moreover, if the original data set
$(\tilde{Z}_1,\dots,\tilde{Z}_{K+1})$
satisfies hierarchical exchangeability (Definition~\ref{def:hier_exch}), then the bootstrapped calibration and test data
\[(\tilde{Z}^{\textnormal{boot}}_{K_0+1},\dots,\tilde{Z}^{\textnormal{boot}}_K,\tilde{Z}^{\textnormal{boot}}_{K+1})\]
satisfies hierarchical exchangeability (Definition~\ref{def:hier_exch}) conditional on the training data $\tilde{Z}_{[K_0]}$.
Consequently, the Repeated Subsampling method satisfies
\[\PP{Y_{\test}\in\Ch(X_{\test})}\geq 1-\alpha.\]
and moreover, if the scores $s(Z_{k,i})$ are distinct almost surely,
it additionally holds that
\[\PP{Y_{\test}\in\Ch(X_{\test})} \leq 1-\alpha + \frac{2}{K_1 + 1}.\]
\end{proposition}
\begin{proof}
The first claim---i.e., that Repeated Subsampling is equivalent to running HCP with the bootstrapped
calibration set---follows directly from the definition of the methods. Next we check hierarchical exchangeability. 
First, for a permutation $\sigma$ of the groups $K_0+1,\dots,K+1$, we observe that the permuted
bootstrapped data set
\[(\tilde{Z}^{\textnormal{boot}}_{\sigma(K_0+1)},\dots,\tilde{Z}^{\textnormal{boot}}_{\sigma(K+1)})\]
is equal, in distribution, to first permuting the original calibration and test data (i.e., replacing 
$(\tilde{Z}_{K_0+1},\dots,\tilde{Z}_{K+1})$ with $(\tilde{Z}_{\sigma(K_0+1)},\dots,\tilde{Z}_{\sigma(K+1)})$,
which is equal in distribution conditional on $\tilde{Z}_{[K_0]}$),
and then sampling data points at random for each group; therefore, it is equal in distribution to the bootstrapped
data $(\tilde{Z}^{\textnormal{boot}}_{K_0+1},\dots,\tilde{Z}^{\textnormal{boot}}_{K+1})$. The second exchangeability property (i.e., equality in distribution after
permuting within a group) follows from the fact that the randomly sampled indices satisfy
$(i_k^{(1)},\dots,i_k^{(B)})\eqd (i_k^{(\sigma(1))},\dots,i_k^{(\sigma(B))})$ for any $\sigma\in\mathcal{S}_B$.
We can now apply Theorem~\ref{thm:HCP} to prove the coverage results.\footnote{For the upper bound on coverage,
note that the footnote in Theorem~\ref{thm:HCP} states that it is sufficient for scores to be distinct across groups; in particular,
the bootstrapped data set may have $i_k^{(b)} = i_k^{(b')}$ leading to repeated scores $s(Z_{k,i_k^{(b)}}) = s(Z_{k,i_k^{(b')}})$
within the $k$th group, but this does not contradict the result since we only need to ensure $s(Z_{k,i_k^{(b)}})\neq s(Z_{k',i_{k'}^{(b')}})$
for distinct groups $k\neq k'$.}
\end{proof}

\section{Additional experimental results: HCP vs conformal prediction}

Here, we provide simulation results comparing the proposed HCP method and standard conformal prediction, to illustrate that conformal prediction fails to achieve valid and tight coverage in the hierarchical data setting, while HCP successfully addresses this issue.

We generate data according to Setting 2 in Section~\ref{section:simulation}, where the conditional variance of $Y$ given $X$ varies across the feature space. We run experiments under three different group size settings:
\begin{enumerate}
    \item $N \equiv 5$,
    \item $N \mid X \sim \textnormal{Poisson}(2X)$,
    \item $N \mid X \sim \textnormal{Poisson}(10-2X)$.
\end{enumerate}
The first setting represents the case where the group sizes are homogeneous, and thus standard conformal prediction is also expected to perform well in the hierarchical data setting. The second setting corresponds to the case where more samples are likely to come from the `difficult region'; unless the heterogeneity in group sizes is accounted for, the resulting prediction set is expected to be conservative. The third setting represents the opposite case. The results shown in Figure~\ref{fig:HCP_CP} confirm that the two procedures behave as expected: In all three settings, HCP successfully addresses the effect of group sizes and attains the target coverage rate accurately, while standard conformal prediction does not.

\begin{figure}[h]
	\begin{center}
		\includegraphics[width=0.85\textwidth]{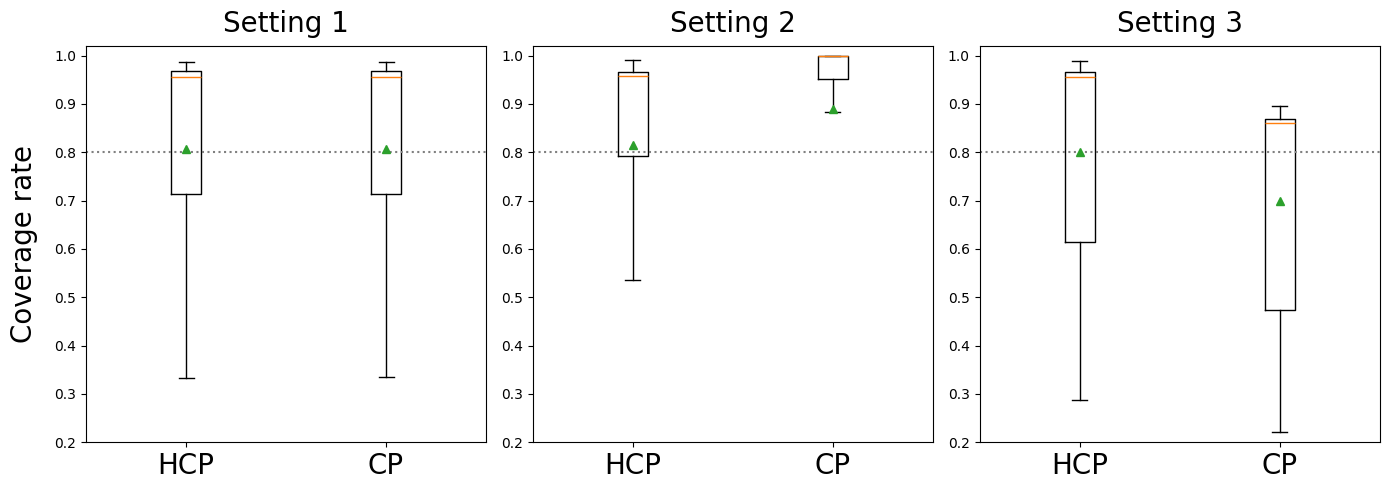}
		\caption{Coverage rates of HCP and split conformal prediction (denoted as CP) under three different group size settings.}
		\label{fig:HCP_CP}
	\end{center}
\end{figure}

\section{Extending HCP to full conformal and jackknife+}\label{Appendix:A}
In Sections~\ref{sec:HCP} and~\ref{sec:HCP2}, we introduced methods that extend split conformal prediction to the setting of hierarchical exchangeability, constructing prediction sets with the marginal coverage guarantee (via the HCP method) and the stronger second-moment coverage guarantee (via the HCP$^2$ method). Similar ideas can be applied to derive extensions of other methods---specifically, full conformal prediction \citep{vovk2005algorithmic} and jackknife+ \citep{barber2021predictive}---to the hierarchical data setting, which allows us to avoid
the statistical cost of data splitting.

\subsection{Extending full conformal prediction}

\subsubsection{Background: full conformal prediction}
In the setting of exchangeable data (without a hierarchical structure), the \emph{full conformal prediction} method \citep{vovk2005algorithmic}
allows us to use the entire training data set for constructing the score function, rather than splitting the data as in split conformal.
Specifically, the method begins with an algorithm $\mathcal{A}: \mathcal{Z}^{n+1}\rightarrow\{\textnormal{functions $\mathcal{Z}\rightarrow \R$}\}$, which
inputs a data set of size $n+1$ and returns a fitted score function $s$; $\mathcal{A}$ is assumed to be symmetric in its input arguments, i.e., permuting
the $n+1$ input data points does not affect $s$.\footnote{The framework also allows for a randomized algorithm $\mathcal{A}$, in which
case the symmetry condition is required to hold in a distributional sense.} Define 
\[s^y = \mathcal{A}\left(Z_1,\dots,Z_n,(X_{\test},y)\right),\]
the fitted score function obtained if $(X_{\test},y)$ were the test point. 
The prediction set is then given by
\[\Ch(X_{\test}) = \left\{y \in \mathcal{Y} : s^y((X_{\test},y)) \leq Q_{1-\alpha}\left(\sum_{i=1}^n \frac{1}{n}\delta_{s^y(Z_i)} + \frac{1}{n}\delta_{+\infty}\right)\right\}.\] 
For exchangeable training and test data, this method offers marginal coverage at level $1-\alpha$ \citep{vovk2005algorithmic}.

\subsubsection{Hierarchical full conformal prediction}
We now define a version of HCP that extends the full conformal method to the hierarchical framework.
 We now begin with an algorithm $\mathcal{A} : \tilde{\mathcal{Z}}^{K+1}\rightarrow \{\textnormal{functions $\mathcal{Z}\rightarrow \R$}\}$,
 again assumed to be symmetric---but now the symmetry follows a hierarchical structure: we assume that,
 for any $\tilde{z}_1,\dots,\tilde{z}_{K+1}$, 
first, it holds that
 \[\mathcal{A}(\tilde{z}_1,\dots,\tilde{z}_{K+1}) =\mathcal{A} (\tilde{z}_{\sigma(1)},\dots,\tilde{z}_{\sigma(K+1)})\]
 for all $\sigma\in\mathcal{S}_{K+1}$,
and, second, 
\[\mathcal{A}(\tilde{z}_1,\dots,\tilde{z}_k,\dots,\tilde{z}_{K+1}) =\mathcal{A}(\tilde{z}_1,\dots,(\tilde{z}_{k,\sigma(1)},\dots,\tilde{z}_{k,\sigma(m)}),\dots,\tilde{z}_{K+1})\]
 for all $k=1,\dots,K+1$ and
all $\sigma\in\mathcal{S}_{\textnormal{length}(\tilde{z}_k)}$.
We  refer to any such algorithm $\mathcal{A}$ as \emph{hierarchically symmetric}.
(For example, if $\mathcal{A}$ simply discards the group information and runs a symmetric procedure
on the data points $(z_{1,1},\dots,z_{1,N_1},\dots,z_{K+1,1},\dots,z_{K+1,N_{K+1}})$, then
the hierarchical symmetry property follows immediately---in this sense, then, requiring hierarchical symmetry
is strictly weaker than the usual symmetry condition.)

Next, define 
\[s^{\tilde{z}} = \mathcal{A}\left(\tilde{Z}_1,\dots,\tilde{Z}_K,\tilde{z}\right)\]
to be the fitted score function obtained if we were to observe an entire \emph{group} of test observations $\tilde{z}$, and let
\[\tilde{C} = \left\{\tilde{z} \in\tilde{\mathcal{Z}} : s^{\tilde{z}}(P_1\tilde{z}) \leq Q_{1-\alpha}\left(\sum_{k=1}^{K} \sum_{i=1}^{N_k} \frac{1}{(K+1)N_k}\cdot \delta_{s^{\tilde{z}}(Z_{k,i})} +\frac{1}{K+1} \cdot \delta_{+\infty} \right)\right\},\]
where $P_1$ denotes the projection to the first component (i.e., if $\tilde{z} = (z_1,\dots,z_m)$, then $P_1(\tilde{z}) = z_1$).
Let \begin{equation}\label{eqn:HCP_fullCP}\Ch(X_{\test}) = \{y \in \R : (X_{\test},y) \in P_1 \tilde{C}\},\end{equation} where $P_1\tilde{C} = \{P_1(\tilde{z}) : \tilde{z}\in\tilde{C}\}$.

Since $\ \mathcal{Z} \cup \mathcal{Z}^2 \cup \dots$ is a space of vectors of arbitrary lengths, this prediction set has heavier computational cost than the full conformal method in standard settings, which already is known to have a high computational cost. It is therefore unlikely to be useful in practical settings; this extension is therefore primarily of theoretical interest, and is intended only to demonstrate that the hierarchical exchangeability framework can be combined with full conformal as well as with split conformal.

Our first theoretical result shows that this method
 provides a marginal coverage guarantee in the hierarchical data setting.
\begin{theorem}\label{thm:HCP_full_conformal}
Suppose that $\tilde{Z}_1,\dots,\tilde{Z}_{K+1}$ satisfies the hierarchical exchangeability property (Definition~\ref{def:hier_exch}), and let $Z_{\test}=(X_{\test},Y_{\test})=\tilde{Z}_{K+1,1}$. Given a hierarchically symmetric algorithm $\mathcal{A}$, the prediction set $\Ch(X_{\textnormal{test}})$ defined above satisfies
\[\PP{Y_{\test} \in \Ch(X_{\test})} \geq 1-\alpha.\]
\end{theorem}

Similarly, we can achieve a bound for the second-moment coverage via full-conformal framework. We extend HCP$^2$ to this setting. 
Define $K^{\geq 2} = \sum_{k=1}^K \one{N_k\geq 2}$, and let
\begin{multline*}
\tilde{C} = \left\{\rule{0cm}{1.2cm}\tilde{z} \in \tilde{\mathcal{Z}} \right. : s^{\tilde{z}}(P_1\tilde{z}) \leq  Q_{1-\alpha^2}\left(\rule{0cm}{1.1cm}\sum_{\substack{k=1,\dots,K\\ N_k\geq 2}}\sum_{1 \leq i < i' \leq N_k}\frac{1}{(K^{\geq 2} +1){N_k\choose 2}}\cdot\delta_{\min\{s^{\tilde{z}}(Z_{k,i}),s^{\tilde{z}}(Z_{k,i'})\}} \right.\\ \left.
+\frac{1}{K^{\geq 2} +1}\cdot\delta_{+\infty}\rule{0cm}{1.1cm}\right) \left. \rule{0cm}{1.2cm}\right\},
\end{multline*}
and then again define $\Ch(X_{\test})$ as in~\eqref{eqn:HCP_fullCP}.  Again, as above, this construction
is computationally extremely expensive and is therefore primarily of theoretical interest.

\begin{theorem}\label{thm:HCP2_full_conformal}
Assume the training data is drawn from the i.i.d.\ model with repeated measurements~\eqref{eqn:model_repeated},
independently for $k=1,\dots,K$, and that the test point $(X_{\test},Y_{\test})$ is drawn
independently from the model~\eqref{eqn:model_repeated_test_point}. Given a symmetric algorithm $\mathcal{A}$, the prediction set $\Ch(X_{\textnormal{test}})$ defined above satisfies
\[\Ep{P_X^{\geq 2}}{\alpha_{\mathcal{D}}(X_{\test})^2} \leq \alpha^2.\]
\end{theorem}

We omit the proofs of these two theorems,
 as they follow from the same arguments as the proofs of Theorem~\ref{thm:HCP} and Theorem~\ref{thm:HCP2} for the split conformal 
 setting.
\subsection{Extending jackknife+}
\subsubsection{Background: jackknife+}
The jackknife+ \citep{barber2021predictive} (closely
related to cross-conformal prediction \citep{vovk2015cross,vovk2018cross}) offers a leave-one-out approach to the problem of distribution-free predictive inference. Like full conformal, the jackknife+ method provides a prediction set without
the loss of accuracy incurred by data splitting. The computational cost of jackknife+ is higher than for split conformal (requiring fitting $K$ many leave-one-group-out models),
but substantially lower than full conformal. 

Rather than working with an arbitrary score function, jackknife+ works specifically with the residual score (i.e., 
$s(x,y) = |y - \widehat\mu(x)|$ for some fitted model $\widehat\mu$), in the setting $\mathcal{Y}\subseteq\R$.
Let $\mathcal{A}: \mathcal{Z}^{n+1}\rightarrow\{\textnormal{functions $\mathcal{X}\rightarrow \R$}\}$, which
inputs a data set of size $n+1$ and returns a fitted regression function $\widehat\mu$; as for full conformal,
$\mathcal{A}$ is assumed to be symmetric in its input arguments. Let 
\[\widehat\mu_{-i} = \mathcal{A}\left(Z_1,\dots,Z_{i-1},Z_{i+1},\dots,Z_n\right)\]
be the $i$th leave-one-out model, and compute the $i$th leave-one-out residual as
\[R_i = |Y_i - \widehat\mu_{-i}(X_i)|.\]
Then the jackknife+ prediction interval is given by
\[\Ch(X_{\test}) = \left[ Q_\alpha'\left(\sum_{i=1}^n \frac{1}{n}\cdot\delta_{\widehat\mu_{-i}(X_{\test}) - R_i} + \frac{1}{n}\cdot\delta_{-\infty}\right), \ 
Q_{1-\alpha}\left(\sum_{i=1}^n \frac{1}{n}\cdot\delta_{\widehat\mu_{-i}(X_{\test}) + R_i} + \frac{1}{n}\cdot\delta_{+\infty}\right)\right],\] 
where for a distribution $P$ on $\R$, $Q_\alpha'(P) = \sup\{t\in\R: \Pp{T\sim P}{T\leq t} \leq \alpha\}$ (compare to the usual definition
of quantile, where we take the infimum rather than the supremum). 
For exchangeable training and test data, the jackknife+ offers marginal coverage at level $1-2\alpha$ (here the factor of 2 is unavoidable,
unless we make additional assumptions) \citep{barber2021predictive}.

\subsubsection{Hierarchical jackknife+}

Next, we extend jackknife+ to the hierarchical setting.
(HCP can also be extended to cross-conformal prediction \citep{vovk2015cross,vovk2018cross} in a similar way, but we omit the details here.)
While computationally more expensive than split conformal based methods, the hierarchical jackknife+ can nonetheless
be viewed as a practical alternative to the computationally infeasible full conformal based method.

To define the hierarchical jackknife+ method, we begin with an algorithm $\mathcal{A}: \tilde{\mathcal{Z}}^{K-1}\rightarrow \{\textnormal{functions $\mathcal{X}\rightarrow\R$}\}$,
which we again assume to be hierarchically symmetric, as for the case of full conformal.
For each $k=1,\dots,K$, define the leave-one-group-out model,
\[\widehat{\mu}_{-k} = \mathcal{A}\left(\tilde{Z}_1,\dots,\tilde{Z}_{k-1},\tilde{Z}_{k+1},\dots,\tilde{Z}_K\right),\]
and define the corresponding residuals for the $k$th group,
\[R_{k,i} = |Y_{k,i} - \widehat\mu_{-k}(X_{k,i})|,\]
for each $i=1,\dots,N_k$. 
The hierarchical jackknife+ prediction interval is then given by
\begin{multline}\label{eqn:HCP_jackknife+}
\Ch(X_{\test}) = \left[\rule{0cm}{0.8cm}Q_\alpha' \left(\sum_{k=1}^K \sum_{i=1}^{N_k} \frac{1}{(K+1)N_k}\cdot\delta_{\muhat_{-k}(X_{\test})-R_{k,i}}+\frac{1}{K+1}\cdot\delta_{-\infty}\right),\right. \\ \left.
Q_{1-\alpha}\left(\sum_{k=1}^K \sum_{i=1}^{N_k} \frac{1}{(K+1)N_k}\cdot \delta_{\muhat_{-k}(X_{\test})+R_{k,i}}+\frac{1}{K+1}\cdot\delta_{+\infty}\right)\rule{0cm}{0.8cm}\right].
\end{multline}

Our first theoretical result shows that this method
 provides a marginal coverage guarantee in the hierarchical data setting.
\begin{theorem}\label{thm:jackknife_marginal}
Suppose that $\tilde{Z}_1,\dots,\tilde{Z}_{K+1}$ satisfies the hierarchical exchangeability property (Definition~\ref{def:hier_exch}), and let  $Z_{\test}=(X_{\test},Y_{\test})=\tilde{Z}_{K+1,1}$. Given a hierarchically symmetric algorithm $\mathcal{A}$, the prediction set $\Ch(X_{\textnormal{test}})$ defined in~\eqref{eqn:HCP_jackknife+} satisfies
\[\PP{Y_{\test} \in \Ch(X_{\test})} \geq 1-2\alpha.\]
\end{theorem}

We can also obtain a second-moment coverage bound by extending HCP$^2$ to the jackknife+.
Define
{\small\begin{multline}\label{eqn:HCP2_jackknife+}
\Ch(X_{\test}) = \left[\rule{0cm}{0.8cm}Q_{\alpha^2}' \left(\sum_{\substack{k=1,\dots,K\\ N_k\geq 2}} \sum_{1 \leq i < i' \leq N_k} \frac{1}{(K^{\geq 2}+1){N_k\choose 2}}\cdot\delta_{\muhat_{-k}(X_{\test})-\min\{R_{k,i},R_{k,i'}\}}+\frac{1}{K^{\geq 2}+1}\cdot\delta_{-\infty}\right),\right. \\ \left.
Q_{1-\alpha^2}\left(\sum_{\substack{k=1,\dots,K\\ N_k\geq 2}} \sum_{1 \leq i < i' \leq N_k} \frac{1}{(K^{\geq 2}+1){N_k\choose 2}}\cdot \delta_{\muhat_{-k}(X_{\test})+\min\{R_{k,i},R_{k,i'}\}}+\frac{1}{K^{\geq 2}+1}\cdot\delta_{+\infty}\right)\rule{0cm}{0.8cm}\right].
\end{multline} 
}

\begin{theorem}\label{thm:jackknife_stronger}
Assume the training data is drawn from the i.i.d.\ model with repeated measurements~\eqref{eqn:model_repeated},
independently for $k=1,\dots,K$, and that the test point $(X_{\test},Y_{\test})$ is drawn
independently from the model~\eqref{eqn:model_repeated_test_point}. Given a symmetric algorithm $\mathcal{A}$, the prediction set $\Ch(X_{\textnormal{test}})$ defined in~\eqref{eqn:HCP2_jackknife+} satisfies
\[\Ep{P_X^{\geq 2}}{\alpha_{\mathcal{D}}(X_{\test})^2} \leq 4\alpha^2.\]
\end{theorem}

We note that, as for the jackknife+ in the non-hierarchical setting,
the theoretical guarantees hold with $2\alpha$ in place of $\alpha$ (and, similarly, $4\alpha^2 = (2\alpha)^2$ in place of $\alpha^2$).
The proofs of these two theorems are given in Appendix~\ref{app:proof_jack_1} and~\ref{app:proof_jack_2}.

\section{Comparing HCP\texorpdfstring{$^2$}{} versus a trivial solution}\label{app:HCP_with_alpha2}

In Section~\ref{sec:trivial}, we discussed a possible trivial solution to the problem of second-moment coverage: simply 
run a method that guarantees marginal coverage, at level $1-\alpha^2$ in place of $1-\alpha$.
Here we  illustrate that this trivial solution is, as expected, highly impractical because it  inevitably provides
extremely wide prediction intervals; we  compare to HCP$^2$ which is able to keep intervals at roughly the same
width as a marginal-coverage method, in settings where the score function is chosen well.

To illustrate this, we  rerun the simulations of Section~\ref{section:simulation}, adding a third method:
HCP run with $\alpha^2$ in place of $\alpha$. All other details of the data and implementation are exactly the same
as in Section~\ref{section:simulation}.\footnote{Code to reproduce this simulation is available at \url{https://github.com/rebeccawillett/Distribution-free-inference-with-hierarchical-data}.} The results are shown in Figure~\ref{fig:sim_hist_residual_app},~\ref{fig:sim_hist_standardized_app}, and Table~\ref{table:marginal_coverage_app}.

\afterpage{\begin{figure}[H]
\begin{center}
\includegraphics[width=0.6\textwidth]{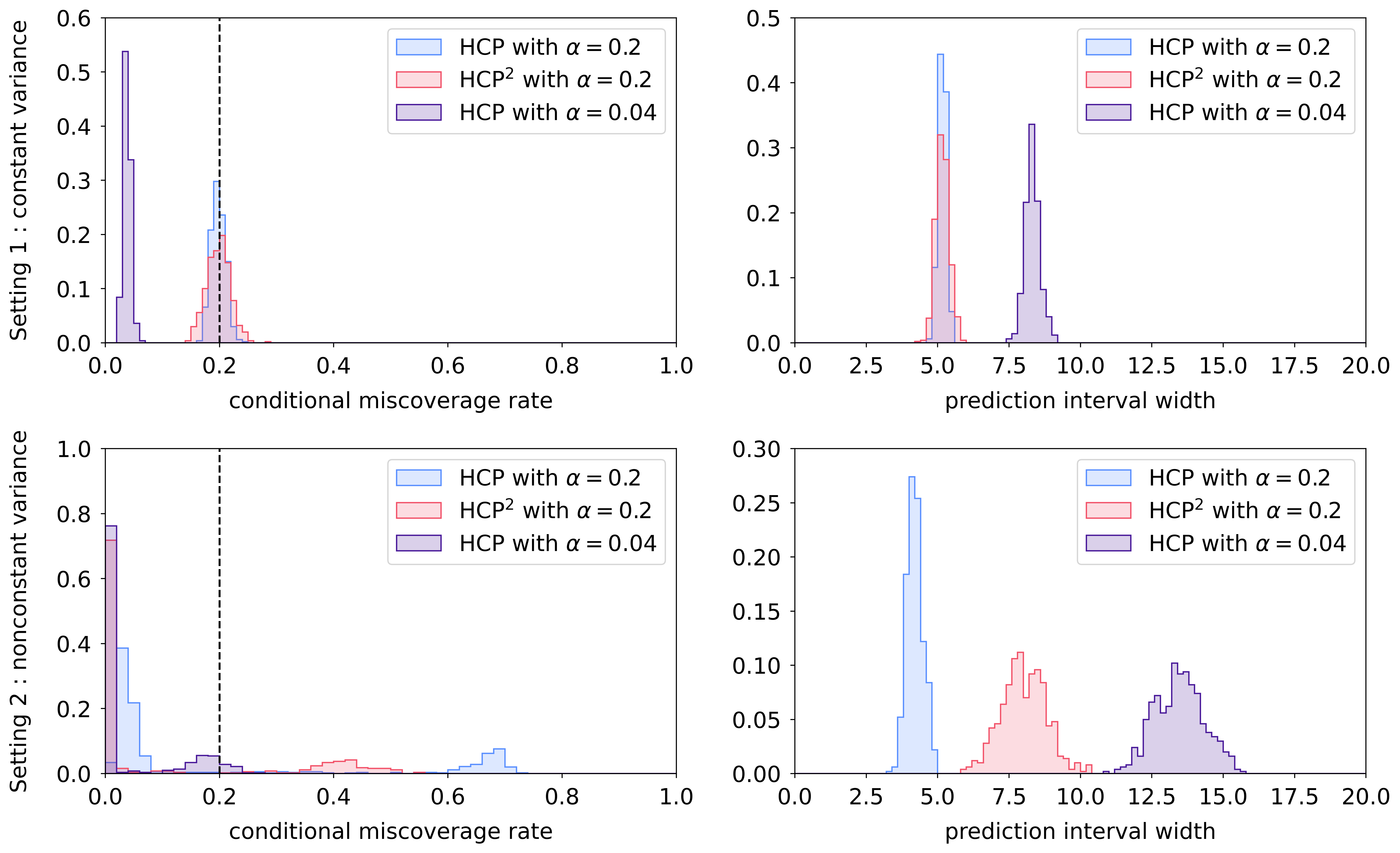}
\caption{Conditional miscoverage rates and widths of HCP and HCP$^2$ with $\alpha=0.2$ and HCP with $\alpha=0.04$, constructed via score $s(x,y) = |y - \widehat{\mu}(x)|$ .}
\label{fig:sim_hist_residual_app}
\end{center}
\end{figure}

\begin{figure}[H]
\begin{center}
\includegraphics[width=0.6\textwidth]{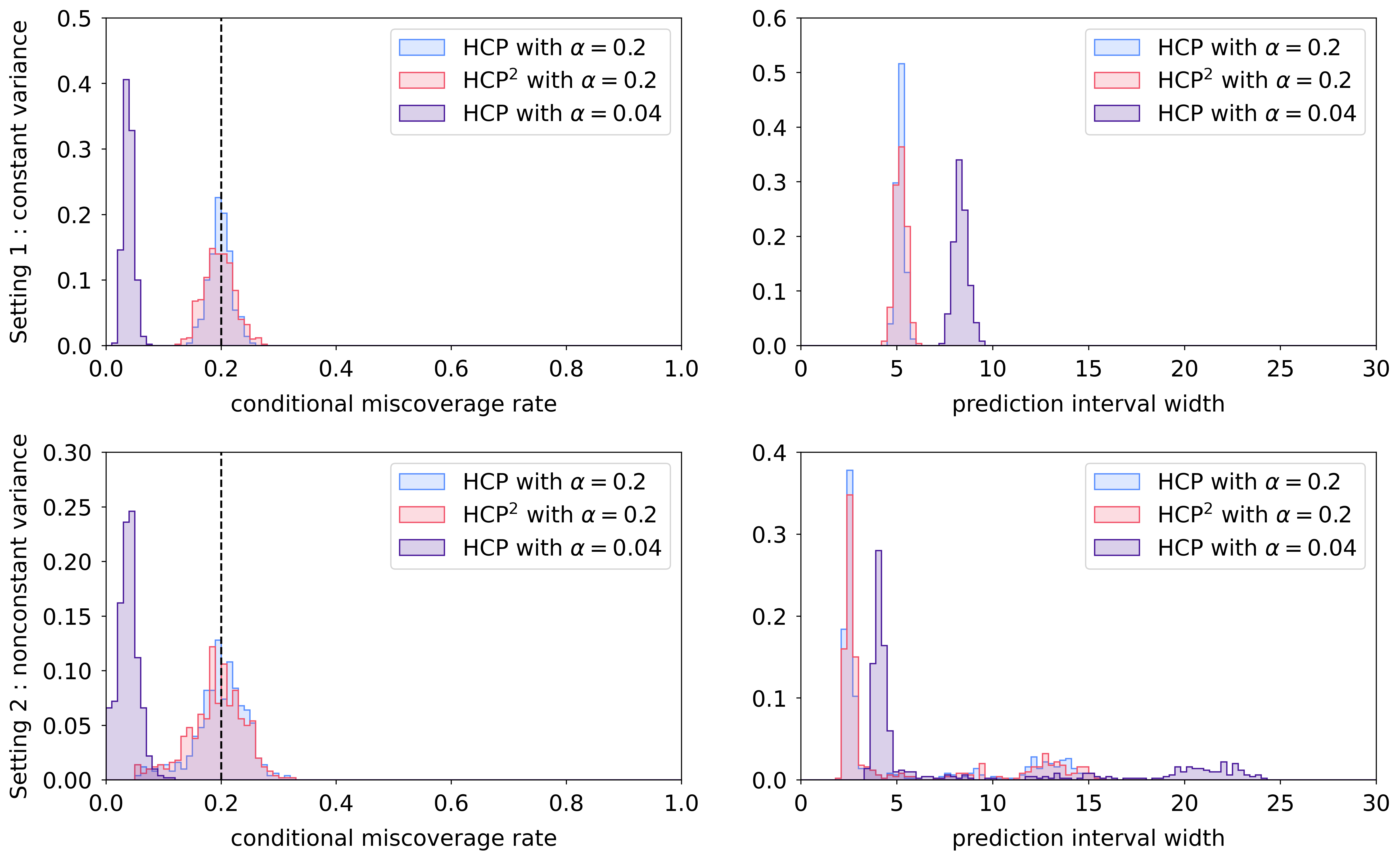}
\caption{Conditional miscoverage rates and widths of HCP and HCP$^2$ with $\alpha=0.2$ and HCP with $\alpha=0.04$, constructed via score $s(x,y) = |y - \widehat{\mu}(x)|/\widehat{\sigma}(x)$ .}
\label{fig:sim_hist_standardized_app}
\end{center}
\end{figure}

\begin{table}[h]
\begin{center} {\footnotesize
\begin{tabular}{cc|cc|cc}
\hline
&&\multicolumn{2}{c|}{$s(x,y) = |y - \widehat{\mu}(x)|$}&\multicolumn{2}{c}{$s(x,y) = |y - \widehat{\mu}(x)|/\widehat\sigma(x)$}\\
&&$\EE{\alpha_{\mathcal{D}}(X_{\test})}$ & $\EE{\alpha_{\mathcal{D}}(X_{\test})^2}$ &$\EE{\alpha_{\mathcal{D}}(X_{\test})}$ & $\EE{\alpha_{\mathcal{D}}(X_{\test})^2}$\\
\hline
\multirow{3}{*}{Setting 1} & HCP with $\alpha = 0.2$ & 0.1979 (0.0006)  & 0.0393 (0.0002)  &  0.1985 (0.0009)  & 0.0398 (0.0003) \\
& HCP$^2$ with $\alpha = 0.2$ & 0.1983 (0.0009)  & 0.0398 (0.0004)  & 0.1965 (0.0012)  & 0.0393 (0.0005) \\
& HCP with $\alpha = 0.04$ & 0.0385 (0.0003)  &  0.0015 (0.0000)  & 0.0391 (0.0004)  & 0.0016 (0.0000) \\
\hline
\multirow{3}{*}{Setting 2} & HCP with $\alpha = 0.2$ & 0.2017 (0.0118)  & 0.1108 (0.0084)  & 0.1982 (0.0020)  & 0.0414 (0.0008) \\
& HCP$^2$ with $\alpha = 0.2$ & 0.1011 (0.0079)  &  0.0413 (0.0034)  & 0.1917 (0.0022)  & 0.0391 (0.0008) \\
& HCP with $\alpha = 0.04$ & 0.0405 (0.0034)  &  0.0073 (0.0007)  & 0.0384 (0.0008)  & 0.0018 (0.0001) \\
\hline
\end{tabular} }
\end{center}
\caption{Mean miscoverage rate (target: $0.2$) and mean second-moment miscoverage rate (target: $0.2^2 = 0.04$), for HCP and HCP$^2$, with standard errors in parentheses. Results are averaged over 500 independent trials.}\label{table:marginal_coverage_app}
\end{table}}

As expected, we see that, for the three cases where the score function is a good match to the data distribution
(namely, the residual score $s(x,y) = |y-\widehat\mu(x)|$ for Setting 1, or the rescaled residual score
$s(x,y) = |y-\widehat\mu(x)|/\widehat\sigma(x)$ for Setting 1 and for Setting 2), HCP$^2$ provides intervals that are similar
in width to HCP, while running the trivial solution (HCP with $\alpha = 0.04$ in place of $\alpha=0.2$) leads to extremely
wide intervals. On the other hand, for the fourth case (the  residual score $s(x,y) = |y-\widehat\mu(x)|$ for Setting 2),
HCP$^2$ shows poor performance, but is still noticeably better than
simply running the trivial solution (HCP with $\alpha = 0.04$ in place of $\alpha=0.2$); the trivial solution is 
still overly conservative, although less so here than in the other three cases.

\section{Proofs}

\subsection{Proof of Theorem~\ref{thm:HCP}}
As described in Section~\ref{sec:intro_setting}, we can define $(X_{\test},Y_{\test}) = (X_{K+1,1},Y_{K+1,1})$,
so that the combined data set $ \tilde{Z}_1,\dots,\tilde{Z}_{K+1}$ satisfies hierarchical exchangeability (Definition~\ref{def:hier_exch}), 
where $\tilde{Z}_k = (Z_{k,1},\dots,Z_{k,N_k})$ for each $k=1,\dots,K+1$ and $Z_{k,i} = (X_{k,i},Y_{k,i})$ for each $i=1,\dots,N_k$.

From this point on, we  perform all calculations
conditional on the training data  $\tilde{Z}_{[K_0]} = (\tilde{Z}_1,\dots,\tilde{Z}_{K_0})$, so that we can treat the score function $s$ as fixed.
Note that, conditional on the training data $\tilde{Z}_{[K_0]}$,
the remaining data $\tilde{Z}_{K_0+1},\dots,\tilde{Z}_K,\tilde{Z}_{K+1}$ (i.e., the calibration and test groups) again 
satisfy hierarchical exchangeability (Definition~\ref{def:hier_exch}).

\paragraph{A new quantile function.} First, define a function
$q:\tilde{\mathcal{Z}}^{K_1+1}\rightarrow\R$ as
\[q_{1-\alpha}(\tilde{z}_1,\dots,\tilde{z}_{K_1+1}) = Q_{1-\alpha}\left(\sum_{k=1}^{K_1+1}\sum_{i=1}^{\textnormal{length}(\tilde{z}_k)}
\frac{1}{(K_1+1)\cdot \textnormal{length}(\tilde{z}_k)}\delta_{s(z_{k,i})}\right),\]
the empirical quantile of a given data set reweighted according to group size (to give equal weight to each group regardless
of length).
By definition of $q_{1-\alpha}(\tilde{Z}_{K_0+1},\dots,\tilde{Z}_{K+1})$ as a weighted quantile, it holds deterministically that
\begin{equation}\label{eqn:q_cover_lowerbd}\frac{1}{K_1+1} \sum_{k=K_0+1}^{K+1}\frac{1}{N_k}\sum_{i=1}^{N_k}\One{s(Z_{k,i})\leq q_{1-\alpha}(\tilde{Z}_{K_0+1},\dots,\tilde{Z}_{K+1})}
\geq 1-\alpha\end{equation}
(see, e.g., \citet[Lemma A1]{harrison2012conservative}), and furthermore, if $s(Z_{k,i})\neq s(Z_{k',i'})$ for any $k\neq k'$,
then
\begin{equation}\label{eqn:q_cover_upperbd}\frac{1}{K_1+1} \sum_{k=K_0+1}^{K+1}\frac{1}{N_k}\sum_{i=1}^{N_k}\One{s(Z_{k,i})\leq q_{1-\alpha}(\tilde{Z}_{K_0+1},\dots,\tilde{Z}_{K+1})}
\leq 1-\alpha + \frac{1}{K_1+1},\end{equation}
where this last step holds since the condition on the scores ensures that the weighted distribution places mass at most $\frac{1}{K_1+1}$ on any value.

Next, we examine the invariance properties of this function $q_{1-\alpha}$.
First, for any $\sigma\in\mathcal{S}_{K_1+1}$, we can verify that
\begin{equation}\label{eqn:q_perm_1}q_{1-\alpha}(\tilde{z}_1,\dots,\tilde{z}_{K_1+1}) = q_{1-\alpha}(\tilde{z}_{\sigma(1)},\dots,\tilde{z}_{\sigma(K_1+1)}) .\end{equation}
Next, for any $k$ consider a permutation $\sigma\in\mathcal{S}_{\textnormal{length}(\tilde{z}_k)}$,
and let $\tilde{z}_k^\sigma = (z_{k,\sigma(1)},\dots,z_{k,\sigma(\textnormal{length}(\tilde{z}_k))})$
be the corresponding permutation of $\tilde{z}_k$. We can see from the definition
that 
\begin{equation}\label{eqn:q_perm_2}q_{1-\alpha}(\tilde{z}_1,\dots,\tilde{z}_{K_1+1}) = q_{1-\alpha}(\tilde{z}_1,\dots,
\tilde{z}_{k-1},\tilde{z}_k^\sigma,\tilde{z}_{k+1},\dots,\tilde{z}_{K_1+1}).\end{equation}

\paragraph{Applying hierarchical exchangeability.}
Recall the definition of hierarchical exchangeability (Definition~\ref{def:hier_exch}): 
by the second part of the definition, for any permutation $\sigma\in\mathcal{S}_{N_{K+1}}$, we have
\[(\tilde{Z}_{K_0+1},\dots,\tilde{Z}_{K+1})\eqd
(\tilde{Z}_{K_0+1},\dots,\tilde{Z}_K,\tilde{Z}_{K+1}^\sigma)\]
conditional on both $N_{K+1}$ and on $\tilde{Z}_{[K_0]}$. Then, conditional on $N_{K+1}$ and $\tilde{Z}_{[K_0]}$, for any $i$,
\begin{multline*}\One{s(Z_{K+1,1})\leq q_{1-\alpha}(\tilde{Z}_{K_0+1},\dots,\tilde{Z}_{K+1})} 
\eqd \One{s(Z_{K+1,\sigma(1)})\leq q_{1-\alpha}(\tilde{Z}_{K_0+1},\dots,\tilde{Z}_K,\tilde{Z}_{K+1}^\sigma)}\\
=\One{s(Z_{K+1,\sigma(1)})\leq q_{1-\alpha}(\tilde{Z}_{K_0+1},\dots,\tilde{Z}_{K+1})} ,\end{multline*}
where the second step holds by~\eqref{eqn:q_perm_2}. In particular, by taking any $\sigma$ with $\sigma(1)=i$, 
\begin{multline*}\PPst{s(Z_{K+1,1})\leq q_{1-\alpha}(\tilde{Z}_{K_0+1},\dots,\tilde{Z}_{K+1})}{\tilde{Z}_{[K_0]},N_{K+1}}\\
=\PPst{s(Z_{K+1,i})\leq q_{1-\alpha}(\tilde{Z}_{K_0+1},\dots,\tilde{Z}_{K+1})}{\tilde{Z}_{[K_0]},N_{K+1}},\end{multline*}
for each $i=1,\dots,N_{K+1}$. After taking an average over all $i$, and then marginalizing over $N_{K+1}$,
\begin{multline*}\PPst{s(Z_{K+1,1})\leq q_{1-\alpha}(\tilde{Z}_{K_0+1},\dots,\tilde{Z}_{K+1})}{\tilde{Z}_{[K_0]}}\\
=\EEst{\frac{1}{N_{K+1}}\sum_{i=1}^{N_{K+1}}\One{s(Z_{K+1,i})\leq q_{1-\alpha}(\tilde{Z}_{K_0+1},\dots,\tilde{Z}_{K+1})}}{\tilde{Z}_{[K_0]}}.\end{multline*}

Furthermore, by the first part of Definition~\ref{def:hier_exch}, we see that
\[(\tilde{Z}_{K_0+1},\dots,\tilde{Z}_{K+1})\eqd (\tilde{Z}_{\sigma(K_0+1)},\dots,\tilde{Z}_{\sigma(K+1)})\]
holds conditional on $\tilde{Z}_{[K_0]}$, for any $\sigma$ that permutes $\{K_0+1,\dots,K+1\}$.
Thus, conditional on $\tilde{Z}_{[K_0]}$, it holds that
\begin{multline*}
\frac{1}{N_{K+1}}\sum_{i=1}^{N_{K+1}}\One{s(Z_{K+1,i})\leq q_{1-\alpha}(\tilde{Z}_{K_0+1},\dots,\tilde{Z}_{K+1})}\\
\eqd 
\frac{1}{N_{\sigma(K+1)}}\sum_{i=1}^{N_{\sigma(K+1)}}\One{s(Z_{\sigma(K+1),i})\leq q_{1-\alpha}(\tilde{Z}_{\sigma(K_0+1)},\dots,\tilde{Z}_{\sigma(K+1)})}\\
=\frac{1}{N_{\sigma(K+1)}}\sum_{i=1}^{N_{\sigma(K+1)}}\One{s(Z_{\sigma(K+1),i})\leq q_{1-\alpha}(\tilde{Z}_{K_0+1},\dots,\tilde{Z}_{K+1})}
\end{multline*}
where the second step applies~\eqref{eqn:q_perm_1}. Then for any $k\in\{K_0+1,\dots,K+1\}$, taking any $\sigma$ with $\sigma(K+1)=k$, we have
\begin{multline*}
\EEst{\frac{1}{N_{K+1}}\sum_{i=1}^{N_{K+1}}\One{s(Z_{K+1,i})\leq q_{1-\alpha}(\tilde{Z}_{K_0+1},\dots,\tilde{Z}_{K+1})}}{\tilde{Z}_{[K_0]}}\\
= \EEst{\frac{1}{N_k}\sum_{i=1}^{N_k}\One{s(Z_{k,i})\leq q_{1-\alpha}(\tilde{Z}_{K_0+1},\dots,\tilde{Z}_{K+1})}}{\tilde{Z}_{[K_0]}}\end{multline*}
and so, taking an average over all $k$,
\begin{multline*}
\EEst{\frac{1}{N_{K+1}}\sum_{i=1}^{N_{K+1}}\One{s(Z_{K+1,i})\leq q_{1-\alpha}(\tilde{Z}_{K_0+1},\dots,\tilde{Z}_{K+1})}}{\tilde{Z}_{[K_0]}}\\
= \EEst{\frac{1}{K_1+1} \sum_{k=K_0+1}^{K+1}\frac{1}{N_k}\sum_{i=1}^{N_k}\One{s(Z_{k,i})\leq q_{1-\alpha}(\tilde{Z}_{K_0+1},\dots,\tilde{Z}_{K+1})}}{\tilde{Z}_{[K_0]}}.\end{multline*}

Combining everything so far, we have shown that
\begin{multline*}
\PPst{s(Z_{K+1,1})\leq q_{1-\alpha}(\tilde{Z}_{K_0+1},\dots,\tilde{Z}_{K+1})}{\tilde{Z}_{[K_0]}}
\\ =  \EEst{\frac{1}{K_1+1} \sum_{k=K_0+1}^{K+1}\frac{1}{N_k}\sum_{i=1}^{N_k}\One{s(Z_{k,i})\leq q_{1-\alpha}(\tilde{Z}_{K_0+1},\dots,\tilde{Z}_{K+1})}}{\tilde{Z}_{[K_0]}}.\end{multline*}
Therefore, applying~\eqref{eqn:q_cover_lowerbd}, we have
\[\PPst{s(Z_{K+1,1})\leq q_{1-\alpha}(\tilde{Z}_{K_0+1},\dots,\tilde{Z}_{K+1})}{\tilde{Z}_{[K_0]}} \geq 1-\alpha.\]
Conversely, if $s(Z_{k,i})\neq s(Z_{k',i'})$ for all $k\neq k'$, almost surely, then by~\eqref{eqn:q_cover_upperbd},
\[\PPst{s(Z_{K+1,1})\leq q_{1-\alpha}(\tilde{Z}_{K_0+1},\dots,\tilde{Z}_{K+1})}{\tilde{Z}_{[K_0]}} \leq 1-\alpha + \frac{1}{K_1+1}.\]

\paragraph{Coverage guarantee: lower bound.} Now we need to relate the function $q_{1-\alpha}$ to the coverage properties of the HCP method.
Recalling the definition~\eqref{eqn:define_HCP} of the HCP prediction interval, we can see that
\[Q_{1-\alpha}\left(\sum_{k=K_0+1}^K\sum_{i=1}^{N_k}\frac{1}{(K_1+1)N_k}\cdot\delta_{s(Z_{k,i})}+\frac{1}{K_1+1}\cdot\delta_{+\infty}\right)
\geq q_{1-\alpha}(\tilde{Z}_{K_0+1},\dots,\tilde{Z}_{K+1}),\]
and so
\[  s(Z_{K+1,1})\leq q_{1-\alpha}(\tilde{Z}_{K_0+1},\dots,\tilde{Z}_{K+1})\quad \Rightarrow \quad Y_{\test}\in \Ch(X_{\test}).\]
Therefore, combining with the work above,
\[\PPst{ Y_{\test}\in \Ch(X_{\test})}{\tilde{Z}_{[K_0]}}
 \geq \PPst{s(Z_{K+1,1}) \leq q_{1-\alpha}(\tilde{Z}_{K_0+1},\dots,\tilde{Z}_{K+1})}{\tilde{Z}_{[K_0]}}
\geq 1-\alpha.\]

\paragraph{Coverage guarantee: upper bound.} Finally, we derive the upper bound on coverage, in the case that
 $s(Z_{k,i})\neq s(Z_{k',i'})$ for all $k\neq k'$, almost surely.
 First, we can observe that
\[
Q_{1-\alpha}\left(\sum_{k=K_0+1}^K\sum_{i=1}^{N_k}\frac{1}{(K_1+1)N_k}\cdot\delta_{s(Z_{k,i})} +\frac{1}{K_1+1}\cdot\delta_{+\infty}\right)
 \leq q_{1-\alpha'}(\tilde{Z}_{K_0+1},\dots,\tilde{Z}_{K+1}),\]
where $\alpha' = \alpha - \frac{1}{K_1+1}$.
Therefore, from the calculations above (applied with $\alpha'$ in place of $\alpha$),
\begin{multline*}\PPst{Y_{\test}\in\Ch(X_{\test})}{\tilde{Z}_{[K_0]}} \leq \PPst{s(Z_{K+1,1}) \leq q_{1-\alpha'}(\tilde{Z}_{K_0+1},\dots,\tilde{Z}_{K+1})}{\tilde{Z}_{[K_0]}}\\
\leq 1-\alpha' + \frac{1}{K_1+1} = 1-\alpha + \frac{2}{K_1+1}.\end{multline*}

\subsection{Proof of Theorem~\ref{thm:HCP2}}
The proof  largely follows the same structure as the proof of Theorem~\ref{thm:HCP}.

First, draw $(X_{K+1},N_{K+1})$ from the distribution of $(X,N)$ conditional on the event $N\geq 2$, and draw
$Y_{K+1,1},\dots,Y_{K+1,N_{K+1}}$ i.i.d.\ from the distribution of $Y\mid X$ at $X=X_{K+1}$.
Condition also on $\mathcal{S} = \{k\in\{K_0+1,\dots,K\} : N_k\geq 2\}$. 
Then, conditional on $\mathcal{S}$ and also on the training data $\tilde{Z}_{[K_0]}$, we see that
\[(\tilde{Z}_k : k\in \mathcal{S}\cup \{K+1\})\]
satisfies hierarchical exchangeability (Definition~\ref{def:hier_exch}).

\paragraph{A new quantile function.} Next, define $\tilde{\mathcal{Z}}^{\geq 2} = \mathcal{Z}^2\cup\mathcal{Z}^3\cup\dots$, and define a function
$q':(\tilde{\mathcal{Z}}^{\geq 2})^{K_1^{\geq 2}+1}\rightarrow\R$ as
\begin{multline*}q'_{1-\alpha^2}(\tilde{z}_1,\dots,\tilde{z}_{K_1^{\geq 2}+1})\\ = Q_{1-\alpha^2}\left(\sum_{k=1}^{K_1^{\geq2}+1}\sum_{1\leq i<i'\leq \textnormal{length}(\tilde{z}_k)}
\frac{1}{(K_1^{\geq2}+1)\cdot {\textnormal{length}(\tilde{z}_k)\choose 2}}\delta_{\min\{s(z_{k,i}),s(z_{k,i'})\}}\right).\end{multline*}
By definition of this weighted quantile, it holds deterministically that
{\small \begin{equation}\label{eqn:q2_cover_lowerbd}\frac{1}{K_1^{\geq2}+1} \sum_{k\in\mathcal{S}\cup\{K+1\}}\frac{1}{{N_k\choose 2}}\sum_{1\leq i<i'\leq N_k}\One{\min\{s(Z_{k,i}),s(Z_{k,i'})\}\leq q'_{1-\alpha^2}(\tilde{Z}_k : k\in \mathcal{S}\cup \{K+1\})}
\geq 1-\alpha\end{equation}}
(see, e.g., \citet[Lemma A1]{harrison2012conservative}).

Next, we examine invariance properties of this function $q'_{1-\alpha^2}$. Similarly to the proof of Theorem~\ref{thm:HCP}, 
 for any $\sigma\in\mathcal{S}_{K_1+1}$, we see that the value of $q'_{1-\alpha^2}(\tilde{Z}_k : k\in \mathcal{S}\cup \{K+1\})$
 is unchanged if we permute groups, or if we permute data points within any group.

\paragraph{Applying hierarchical exchangeability.}
Following a similar argument to the proof of Theorem~\ref{thm:HCP}, we can verify that, since
the data satisfies hierarchical exchangeability (Definition~\ref{def:hier_exch}), and the function $q'_{1-\alpha^2}$ satisfies the invariance properties
described above, 
we have
\begin{multline*}
\PPst{\min\{s(Z_{K+1,1}),s(Z_{K+1,2})\}\leq q'_{1-\alpha^2}(\tilde{Z}_k : k\in \mathcal{S}\cup \{K+1\})}{\tilde{Z}_{[K_0]},\mathcal{S}}
\\ =  \text{\small$\EEst{\frac{1}{K_1^{\geq 2}+1} \sum_{k\in\mathcal{S}\cup\{K+1\}}\frac{1}{{N_k\choose 2}}\sum_{1\leq i<i'\leq N_k}\!\!\!\One{\min\{s(Z_{k,i}),s(Z_{k,i'})\}\leq q'_{1-\alpha^2}(\tilde{Z}_k : k\in \mathcal{S}\cup \{K+1\})}}{\tilde{Z}_{[K_0]},\mathcal{S}}$}.\end{multline*}
Therefore, applying~\eqref{eqn:q2_cover_lowerbd}, we have
\[\PPst{\min\{s(Z_{K+1,1}),s(Z_{K+1,2})\}\leq q'_{1-\alpha^2}(\tilde{Z}_k : k\in \mathcal{S}\cup \{K+1\})}{\tilde{Z}_{[K_0]},\mathcal{S}} \geq 1-\alpha^2.\]

\paragraph{Coverage guarantee.} Now we need to relate the function $q'_{1-\alpha^2}$ to the coverage properties of the HCP$^2$ method.
For each $i=1,2$, we have
$Y_{K+1,i}\in \Ch(X_{K+1})$ if and only if \[s(Z_{K+1,i}) \leq 
Q_{1-\alpha^2}\left(\sum_{\substack{k=K_0+1,\dots,K\\ N_k\geq 2}}\sum_{1 \leq i < i' \leq N_k}\frac{1}{(K_1^{\geq 2} +1){N_k\choose 2}}\cdot\delta_{\min\{s(Z_{k,i}),s(Z_{k,i'})\}}+\frac{1}{K_1^{\geq 2}+1}\cdot\delta_{+\infty}\right),\]
by definition of the method. Deterministically, it holds that this quantile on the right hand side is $\geq q'_{1-\alpha^2}(\tilde{Z}_k : k\in \mathcal{S}\cup \{K+1\})$, and so,
\[s(Z_{K+1,i})\leq q'_{1-\alpha^2}(\tilde{Z}_k : k\in \mathcal{S}\cup \{K+1\}) \quad \Rightarrow \quad Y_{K+1,i}\in \Ch(X_{K+1}).\]
In particular,
\begin{multline*}\PPst{\{Y_{K+1,1}\in \Ch(X_{K+1})\}\cup \{Y_{K+1,2}\in \Ch(X_{K+1})\}}{\tilde{Z}_{[K_0]},\mathcal{S}}\\\geq
\PPst{\min\{s(Z_{K+1,1}),s(Z_{K+1,2})\} \leq q'_{1-\alpha^2}(\tilde{Z}_k : k\in \mathcal{S}\cup \{K+1\})}{\tilde{Z}_{[K_0]},\mathcal{S}}\geq 1-\alpha^2,\end{multline*}
where the last step holds by the work above. Equivalently,
\[\PPst{Y_{K+1,1},Y_{K+1,2}\not\in \Ch(X_{K+1})}{\tilde{Z}_{[K_0]},\mathcal{S}} \leq \alpha^2.\]
Next, since $Y_{K+1,1},Y_{K+1,2}$ are i.i.d.\ draws from the distribution $P_{Y|X}$ at $X=X_{K+1}$ (and are independent of the training and calibration data, by 
the hierarchical i.i.d.\ model), 
\begin{multline*}\PPst{Y_{K+1,1},Y_{K+1,2}\not\in \Ch(X_{K+1})}{\tilde{Z}_{[K_0]},\mathcal{S}}\\
= \EEst{\PPst{Y_{K+1,1},Y_{K+1,2}\not\in \Ch(X_{K+1})}{\tilde{Z}_{[K]},X_{K+1}}}{\tilde{Z}_{[K_0]},\mathcal{S}}\\
= \EEst{\PPst{Y_{K+1,1}\not\in \Ch(X_{K+1})}{\tilde{Z}_{[K]},X_{K+1}}^2}{\tilde{Z}_{[K_0]},\mathcal{S}}.\end{multline*}
Combining with the above, and marginalizing, we have
\[\alpha^2 \geq \EE{\PPst{Y_{K+1,1}\not\in \Ch(X_{K+1})}{\tilde{Z}_{[K]},X_{K+1}}^2} = 
\EE{\alpha_{\mathcal{D}}(X_{K+1})^2},\]
where the last step holds by definition of $\alpha_{\mathcal{D}}(x)$. Recalling that $X_{\test} = X_{K+1}$ by construction, this completes the proof of 
the upper bound on $\EE{\alpha_{\mathcal{D}}(X_{\test})^2}$.

\paragraph{The upper bound.}
In the case that $s(Z_{k,i})\neq s(Z_{k',i'})$ for all $k\neq k'$, almost surely,
proving that 
\[\EEst{\alpha_{\mathcal{D}}(X_{K+1})^2}{K_1^{\geq 2}} \geq \alpha^2 - \frac{2}{K_1^{\geq 2} + 1}\]
follows from an argument that is very similar to the proof of the analogous result in Theorem~\ref{thm:HCP}, so we omit the details.
Finally, we can verify that
\[\EE{ \frac{1}{K_1^{\geq 2} + 1}}  \leq \frac{1}{(K_1 + 1)\cdot p_{\geq 2}}\]
by using the fact that $K_1^{\geq 2}\sim\textnormal{Binomial}(K_1,p_{\geq 2})$.

\subsection{Proof of Theorem~\ref{thm:jackknife_marginal}}\label{app:proof_jack_1}

The proof extends the idea of the proof for the original jackknife+ \citep{barber2021predictive}. Similarly to the proof for the split conformal-based methods we look into the coverage for $Y_{K+1,1}$. For each $1 \leq k \neq k' \leq K+1$, let $\tilde{\mu}_{-(k,k')} =\tilde{\mu}_{-(k',k)} = \mathcal{A}((\tilde{Z}_\ell)_{\ell \in [K+1]\backslash\{k,k'\}})$ be the estimator from the expanded data set $\tilde{Z}_1,\dots, \tilde{Z}_K, \tilde{Z}_{K+1}$ after excluding the $k$-th and $k'$-th groups of observations, and define $R_{(k,i),k'}$ by
\begin{align*}
R_{(k,i),k'} = 
\begin{cases}
|Y_{k,i} - \mutil_{-(k,k')}(X_{k,i})|, &k\neq k'\\
+\infty, &k=k'
\end{cases}
\end{align*}
for $1 \leq i \leq N_k$. Next, define
\[A_{(k,i),(k',i')} = \One{R_{(k,i),k'} > R_{(k',i),k}},\]
and note that for any $(k,i),(k',i')$, we must have $A_{(k,i),(k',i')} +A_{(k',i'),(k,i)} \leq 1$.
Define also
\[A_{(k,i),\bullet} = \sum_{k'=1}^{K+1} \sum_{i'=1}^{N_{k'}} \frac{1}{(K+1)N_{k'}} \cdot A_{(k,i),(k',i')}.\]
Write $A$ to denote the array of all $A_{(k,i),(k',i')}$'s:
\[A = \{A_{(k,i),(k',i')} : 1 \leq k,k' \leq K+1, 1 \leq i \leq N_k, 1 \leq i' \leq N_{k'}\}.\]
Finally, define the set $S(A)$ by
\[S(A) = \{(k,i) : A_{(k,i),\bullet} \geq 1-\alpha\},\]
and let
\[s(A) = \sum_{(k,i) \in S(A)} \frac{1}{(K+1)N_k}.\]
To bound $s(A)$, we have
\begin{align*}
&(1-\alpha)s(A) = \sum_{(k,i)\in S(A)} \frac{1-\alpha}{(K+1)N_k}  \\
&\leq \sum_{(k,i)\in S(A)} \frac{1}{(K+1)N_k}A_{(k,i),\bullet}\\
&= \sum_{(k,i)\in S(A)} \frac{1}{(K+1)N_k} \sum_{k'=1}^{K+1}\sum_{i'=1}^{N_{k'}} \frac{1}{(K+1)N_{k'}} A_{(k,i),(k',i')}\\
&= \sum_{(k,i)\in S(A)}\sum_{(k',i')\in S(A)} \frac{1}{(K+1)N_k} \frac{1}{(K+1)N_{k'}} A_{(k,i),(k',i')}\\
&\hspace{1in}+ \sum_{(k,i)\in S(A)}\sum_{(k',i')\not\in S(A)} \frac{1}{(K+1)N_k} \frac{1}{(K+1)N_{k'}} A_{(k,i),(k',i')}\\
&= \frac{1}{2} \sum_{(k,i)\in S(A)}\sum_{(k',i')\in S(A)} \frac{1}{(K+1)N_k} \frac{1}{(K+1)N_{k'}} \left(A_{(k,i),(k',i')} + A_{(k',i'),(k,i)}\right)\\
&\hspace{1in}+ \sum_{(k,i)\in S(A)}\sum_{(k',i')\not\in S(A)} \frac{1}{(K+1)N_k} \frac{1}{(K+1)N_{k'}} A_{(k,i),(k',i')}\\
&\leq \frac{1}{2} \sum_{(k,i)\in S(A)}\sum_{(k',i')\in S(A)} \frac{1}{(K+1)N_k} \frac{1}{(K+1)N_{k'}} + \sum_{(k,i)\in S(A)}\sum_{(k',i')\not\in S(A)} \frac{1}{(K+1)N_k} \frac{1}{(K+1)N_{k'}} \\
&= - \frac{1}{2} \sum_{(k,i)\in S(A)}\sum_{(k',i')\in S(A)} \frac{1}{(K+1)N_k} \frac{1}{(K+1)N_{k'}} + \sum_{(k,i)\in S(A)}\sum_{k'=1}^{K+1}\sum_{i'=1}^{N_{k'}} \frac{1}{(K+1)N_k} \frac{1}{(K+1)N_{k'}} \\
&= - \frac{1}{2}\left(\sum_{(k,i)\in S(A)}\frac{1}{(K+1)N_k}\right)^2 + \sum_{(k,i)\in S(A)}\frac{1}{(K+1)N_k}\\
&=  - \frac{1}{2} s(A)^2 + s(A),
\end{align*}
which proves $s(A)\leq 2\alpha$.

Next, since $(\tilde{Z}_1,\dots,\tilde{Z}_{K+1})$  satisfies hierarchical exchangeability (Definition~\ref{def:hier_exch}), by a similar argument
as in the proof of Theorem~\ref{thm:HCP}, we have
\[\PPst{(K+1 ,1) \in S(A)}{N_{K+1} = m} = \PPst{(K+1 ,i) \in S(A)}{N_{K+1} = m}\]
for all $i=1,2,\dots, m$, where $m$ is any positive integer with $\PP{N_{K+1} = m}>0$, and
\[\EE{\frac{1}{N_{K+1}}\cdot \sum_{i=1}^{N_{K+1}} \One{(K+1,i) \in S(A)}} = \EE{\frac{1}{N_k}\cdot \sum_{i=1}^{N_k} \One{(k,i) \in S(A)}}\]
for all $k=1,2,\dots,K$. It follows that
\begin{align*}
\PP{(K+1,1) \in S(A)} &= \EE{\PPst{(K+1,1) \in S(A)}{N_{K+1}}}\\
&= \EE{\frac{1}{N_{K+1}}\cdot\sum_{i=1}^{N_{K+1}}\PPst{(K+1,i) \in S(A)}{N_{K+1}}}\\
&= \EE{\frac{1}{N_{K+1}}\cdot \sum_{i=1}^{N_{K+1}} \One{(K+1,i) \in S(A)}}\\
&= \EE{\frac{1}{K+1}\sum_{k=1}^{K+1} \frac{1}{N_k}\cdot \sum_{i=1}^{N_k} \One{(k,i) \in S(A)}}\\
&= \EE{\sum_{(k,i) \in S(A)}  \frac{1}{(K+1)N_k}}
= \EE{s(A)}
\leq 2\alpha,
\end{align*}
where the last inequality holds since we have shown that $s(A) \leq 2\alpha$ holds deterministically. 

Now suppose $Y_{K+1,1} \notin \Ch(X_{K+1,1})$. This implies
that either
\[Y_{K+1,1} < Q_\alpha' \left(\sum_{k=1}^K \sum_{i=1}^{N_k} \frac{1}{(K+1)N_k}\cdot\delta_{\muhat_{-k}(X_{K+1,1})-R_{k,i}}+\frac{1}{K+1}\cdot\delta_{-\infty}\right),\]
or
\[Y_{K+1,1} >  Q_{1-\alpha}\left(\sum_{k=1}^K \sum_{i=1}^{N_k} \frac{1}{(K+1)N_k}\cdot \delta_{\muhat_{-k}(X_{K+1,1})+R_{k,i}}+\frac{1}{K+1}\cdot\delta_{+\infty}\right).\]
From the definition of $Q_\alpha'$ and $Q_{1-\alpha}$, we then have
\[\sum_{k=1}^K \sum_{i=1}^{N_k} \frac{1}{(K+1)N_k}\cdot\One{\muhat_{-k}(X_{K+1,1})-R_{k,i} > Y_{K+1,1}}  \geq1-\alpha,\]
or
\[\sum_{k=1}^K \sum_{i=1}^{N_k} \frac{1}{(K+1)N_k}\cdot\One{\muhat_{-k}(X_{K+1,1})+R_{k,i} < Y_{K+1,1}}  \geq1- \alpha.\]
In either case, we have
\begin{align*}
1-\alpha &\leq  \sum_{k=1}^K \sum_{i=1}^{N_k} \frac{1}{(K+1)N_k}\cdot\One{ |Y_{K+1,1}-\muhat_{-k}(X_{K+1,1})| >R_{k,i}}\\
&= \sum_{k=1}^K \sum_{i=1}^{N_k} \frac{1}{(K+1)N_k}\cdot\One{R_{(K+1,1),k} > R_{(k,i),K+1}}\\
&= \sum_{k=1}^K \sum_{i=1}^{N_k}  \frac{1}{(K+1)N_k}\cdot A_{(K+1,1),(k,i)}
= \sum_{k=1}^{K+1} \sum_{i=1}^{N_k}  \frac{1}{(K+1)N_k}\cdot A_{(K+1,1),(k,i)}
= A_{(K+1,1),\bullet}.
\end{align*}
Hence, we have shown that $Y_{K+1,1} \notin \Ch(X_{K+1,1})$ implies $(K+1,1) \in S(A)$, and so
\[\PP{Y_{K+1,1} \notin \Ch(X_{K+1,1})} \leq \PP{(K+1,1) \in S(A)} \leq 2\alpha.\]

\subsection{Proof of Theorem~\ref{thm:jackknife_stronger}}\label{app:proof_jack_2}
For $1 \leq k\neq k' \leq K+1$ with $N_k, N_{k'} \geq 2$, $1 \leq i_1 < i_2 \leq N_k$, and $1 \leq i_1' < i_2' \leq N_{k'}$, define $\tilde{\mu}_{-(k,k')}$ as in the proof of Theorem~\ref{thm:jackknife_marginal}, and then define
\begin{align*}
R_{(k,i_1,i_2),k'} = 
\begin{cases}
\min\{|Y_{k,i_1} - \mutil_{-(k,k')}(X_i)|,|Y_{k,i_2} - \mutil_{-(k,k')}(X_i)|\}, & k \neq k'\\
+\infty, &k=k'
\end{cases}
\end{align*}
and
\[A_{(k,i_1,i_2),(k',i_1',i_2')} = \One{R_{(k,i_1,i_2),k'} > R_{(k',i_1',i_2'),k}},\]
and note that for any $(k,i_1,i_2),(k',i'_1,i'_2)$, we must have $A_{(k,i_1,i_2),(k',i_1',i_2')}  + A_{(k',i_1',i_2'),(k,i_1,i_2)} \leq 1$.
Next, let
\[S(A) = \{(k,i_1,i_2) : A_{(k,i_1,i_2),\bullet} \geq 1-2\alpha^2,1 \leq k \leq K+1, N_k \geq 2, 1 \leq i_1 < i_2 \leq N_k\}\]
where
\[A_{(k,i_1,i_2),\bullet} = \sum_{\substack{1 \leq k'\leq K+1 \\ N_{k'} \geq 2}} \sum_{1 \leq i_1' < i_2' \leq N_{k'}} \frac{1}{{\tilde{K}^{\geq 2}}{N_{k'} \choose 2}}\cdot A_{(k,i_1,i_2),(k',i_1',i_2')}\]
and 
\[\tilde{K}^{\geq 2} = \sum_{k=1}^{K+1} \One{N_k \geq 2},\]
which we assume to be nonzero.
Let
\[s(A) = \sum_{(k,i_1,i_2) \in S(A)} \frac{1}{{\tilde{K}^{\geq 2}}{N_k \choose 2}}.\]
We then calculate
\begin{align*}
&(1-2\alpha^2)s(A) = \sum_{(k,i_1,i_2)\in S(A)} \frac{1-2\alpha^2}{{\tilde{K}^{\geq 2}}{N_k \choose 2}}  \\
&\leq \sum_{(k,i_1,i_2)\in S(A)} \frac{1}{{\tilde{K}^{\geq 2}}{N_k \choose 2}} A_{(k,i_1,i_2),\bullet}\\
&= \sum_{(k,i_1,i_2)\in S(A)} \frac{1}{{\tilde{K}^{\geq 2}}{N_k \choose 2}}  \sum_{\substack{1 \leq k'\leq K+1 \\ N_{k'} \geq 2}} \sum_{1 \leq i_1' < i_2' \leq N_{k'}} \frac{1}{{\tilde{K}^{\geq 2}}{N_{k'} \choose 2}}\cdot A_{(k,i_1,i_2),(k',i_1',i_2')}\\
&= \sum_{(k,i_1,i_2)\in S(A)} \sum_{(k',i'_1,i'_2)\in S(A)} \frac{1}{{\tilde{K}^{\geq 2}}{N_k \choose 2}}  \frac{1}{{\tilde{K}^{\geq 2}}{N_{k'} \choose 2}} A_{(k,i_1,i_2),(k',i_1',i_2')} \\ 
&\hspace{1in}
+  \sum_{(k,i_1,i_2)\in S(A)} \sum_{\substack{(k',i'_1,i'_2)\not\in S(A)\\N_{k'}\geq 2}} \frac{1}{{\tilde{K}^{\geq 2}}{N_k \choose 2}}  \frac{1}{{\tilde{K}^{\geq 2}}{N_{k'} \choose 2}} A_{(k,i_1,i_2),(k',i_1',i_2')}\\
&= \frac{1}{2}\sum_{(k,i_1,i_2)\in S(A)} \sum_{(k',i'_1,i'_2)\in S(A)} \frac{1}{{\tilde{K}^{\geq 2}}{N_k \choose 2}}  \frac{1}{{\tilde{K}^{\geq 2}}{N_{k'} \choose 2}} \left(A_{(k,i_1,i_2),(k',i_1',i_2')}  + A_{(k',i_1',i_2'),(k,i_1,i_2)} \right)\\ 
&\hspace{1in}
+  \sum_{(k,i_1,i_2)\in S(A)} \sum_{\substack{(k',i'_1,i'_2)\not\in S(A)\\N_{k'}\geq 2}} \frac{1}{{\tilde{K}^{\geq 2}}{N_k \choose 2}}  \frac{1}{{\tilde{K}^{\geq 2}}{N_{k'} \choose 2}} A_{(k,i_1,i_2),(k',i_1',i_2')}\\
&\leq \frac{1}{2}\sum_{(k,i_1,i_2)\in S(A)} \sum_{(k',i'_1,i'_2)\in S(A)} \frac{1}{{\tilde{K}^{\geq 2}}{N_k \choose 2}}  \frac{1}{{\tilde{K}^{\geq 2}}{N_{k'} \choose 2}} 
+  \sum_{(k,i_1,i_2)\in S(A)} \sum_{\substack{(k',i'_1,i'_2)\not\in S(A)\\N_{k'}\geq 2}} \frac{1}{{\tilde{K}^{\geq 2}}{N_k \choose 2}}  \frac{1}{{\tilde{K}^{\geq 2}}{N_{k'} \choose 2}} \\
&=- \frac{1}{2}\sum_{(k,i_1,i_2)\in S(A)} \sum_{(k',i'_1,i'_2)\in S(A)} \frac{1}{{\tilde{K}^{\geq 2}}{N_k \choose 2}}  \frac{1}{{\tilde{K}^{\geq 2}}{N_{k'} \choose 2}} 
+  \sum_{(k,i_1,i_2)\in S(A)} \sum_{\substack{1 \leq k'\leq K+1 \\ N_{k'} \geq 2}} \sum_{1 \leq i_1' < i_2' \leq N_{k'}}\frac{1}{{\tilde{K}^{\geq 2}}{N_k \choose 2}}  \frac{1}{{\tilde{K}^{\geq 2}}{N_{k'} \choose 2}} \\
&=- \frac{1}{2}\left(\sum_{(k,i_1,i_2)\in S(A)}  \frac{1}{{\tilde{K}^{\geq 2}}{N_k \choose 2}}  \right)^2
+  \sum_{(k,i_1,i_2)\in S(A)} \frac{1}{{\tilde{K}^{\geq 2}}{N_k \choose 2}} \\
&=  - \frac{1}{2} s(A)^2 + s(A),
\end{align*}
which proves $s(A)\leq 4\alpha^2$.

Next, by an analogous argument as in the proof of Theorem~\ref{thm:jackknife_marginal},
the hierarchical exchangeability of the data implies that
\begin{multline*}
\PPst{(K+1,1,2) \in S(A)}{N_{K+1} \geq 2}\\
= \EEst{\frac{1}{\tilde{K}^{\geq 2}}\sum_{\substack{1 \leq k\leq K+1 \\ N_k \geq 2}} \frac{1}{{N_k\choose 2}}\cdot \sum_{1 \leq i_1 < i_2 \leq N_k} \One{(k,i_1,i_2) \in S(A)}}{N_{K+1} \geq 2} \leq 4\alpha^2,\end{multline*}
where the last step holds by our deterministic bound on $s(A)$ calculated above.
Our next step is to verify that
\begin{equation}\label{eqn:step_both_Ys}Y_{K+1,1},Y_{K+1,2} \notin \Ch(X_{K+1}) \quad \Rightarrow \quad (K+1,1,2)\in S(A).\end{equation}
To see this, by definition of $\Ch(X_{K+1})$ we see that, for each $i=1,2$, $Y_{K+1,i}\not\in \Ch(X_{K+1})$ implies that
\[1-\alpha^2 \leq \sum_{\substack{1\leq k\leq K\\ N_k\geq 2}} \sum_{1\leq i_1<i_2\leq N_k} \frac{1}{\tilde{K}^{\geq 2}{N_k\choose 2}}\cdot\One{ |Y_{K+1,i} -  \muhat_{-k}(X_{K+1})|> \min\{R_{k,i_1},R_{k,i_2}\}}.\]
Summing over $i=1,2$, then, we have
\begin{align*}
&2(1-\alpha^2)\\
&\leq \sum_{\substack{1\leq k\leq K\\ N_k\geq 2}} \sum_{1\leq i_1<i_2\leq N_k} \frac{1}{\tilde{K}^{\geq 2}{N_k\choose 2}}\cdot \sum_{i=1,2}\One{ |Y_{K+1,i} -  \muhat_{-k}(X_{K+1})|> \min\{R_{k,i_1},R_{k,i_2}\}}\\
&\leq \sum_{\substack{1\leq k\leq K\\ N_k\geq 2}} \sum_{1\leq i_1<i_2\leq N_k} \frac{1}{\tilde{K}^{\geq 2}{N_k\choose 2}}\cdot \left( 1 + \One{ \min_{i=1,2}|Y_{K+1,i} -  \muhat_{-k}(X_{K+1})|> \min\{R_{k,i_1},R_{k,i_2}\}}\right)\\
&\leq 1 + \sum_{\substack{1\leq k\leq K\\ N_k\geq 2}} \sum_{1\leq i_1<i_2\leq N_k} \frac{1}{\tilde{K}^{\geq 2}{N_k\choose 2}}\cdot  \One{ \min_{i=1,2}|Y_{K+1,i} -  \muhat_{-k}(X_{K+1})|> \min\{R_{k,i_1},R_{k,i_2}\}}\\
&= 1 + \sum_{\substack{1\leq k\leq K\\ N_k\geq 2}} \sum_{1\leq i_1<i_2\leq N_k} \frac{1}{\tilde{K}^{\geq 2}{N_k\choose 2}}\cdot A_{(K+1,1,2),(k,i_1,i_2)}\\
&= 1 + A_{(K+1,1,2),\bullet}.
\end{align*}
This verifies~\eqref{eqn:step_both_Ys}, and so we have
\[\PPst{Y_{K+1,1},Y_{K+1,2} \notin \Ch(X_{K+1})}{N_{K+1} \geq 2} \leq \PPst{(K+1,1,2) \in S(A)}{N_{K+1} \geq 2} \leq 4 \alpha^2.\]

To  complete the proof, we calculate
\begin{align*}
&\PPst{Y_{K+1,1},Y_{K+1,2} \notin \Ch(X_{K+1})}{N_{K+1} \geq 2}\\
= &\EEst{\PPst{Y_{K+1,1},Y_{K+1,2} \notin \Ch(X_{K+1})}{\tilde{Z}_{[K]}, X_{K+1}, N_{K+1} \geq 2}}{N_{K+1} \geq 2}\\
= &\EEst{\PPst{Y_{K+1,1},Y_{K+1,2} \notin \Ch(X_{K+1})}{\tilde{Z}_{[K]}, X_{K+1}}}{N_{K+1} \geq 2}\\
= &\EEst{\PPst{Y_{K+1,1} \notin \Ch(X_{K+1})}{\tilde{Z}_{[K]}, X_{K+1}}^2}{N_{K+1} \geq 2}\\
= &\EEst{\alpha_{\mathcal{D}}(X_{K+1})^2}{N_{K+1} \geq 2}\\
= &\Ep{P_X^{\geq 2}}{\alpha_{\mathcal{D}}(X_{K+1})^2},
\end{align*}
where the last two steps hold by definition of $\alpha_{\mathcal{D}}(X_{K+1})$ and of $P_X^{\geq 2}$. Therefore,
\[\Ep{P_X^{\geq 2}}{\alpha_{\mathcal{D}}(X_{\test})^2} = \Ep{P_X^{\geq 2}}{\alpha_{\mathcal{D}}(X_{K+1})^2} \leq 4\alpha^2.\]

\end{document}